\documentclass{tglat2e}
\usepackage{zref-xr,hyperref,pb-diagram,amssymb,epic,eepic,verbatim,graphicx,graphics,epsfig,psfrag}
\hypersetup{colorlinks}
\zexternaldocument{qkirwan2}
\zexternaldocument{qkirwan3}

\usepackage{color}

\definecolor{darkred}{rgb}{0.5,0,0}
\definecolor{darkgreen}{rgb}{0,0.5,0}
\definecolor{darkblue}{rgb}{0,0,0.5}

\hypersetup{ colorlinks,
linkcolor=darkblue,
filecolor=darkgreen,
urlcolor=darkred,
citecolor=darkblue }

\renewcommand\theenumi{\alph{enumi}}
\renewcommand\theenumii{\roman{enumii}}
\theoremstyle{plain}
\newtheorem{theorem}{Theorem}[section]
\newtheorem{lemma}[theorem]{Lemma}

\theoremstyle{remark}
\newtheorem{remark}[theorem]{Remark}

\newcommand\A{\mathcal{A}}
\newcommand\M{\mathcal{M}}

\renewcommand\M{\mathcal{M}}

\newcommand{\K}{\mathcal{K}}

\newcommand{\J}{\mathcal{J}}

\newcommand{\R}{\mathbb{R}}
\renewcommand{\H}{\mathbb{H}}

\newcommand{\C}{\mathbb{C}}

\newcommand{\Z}{\mathbb{Z}}
\newcommand{\Q}{\mathbb{Q}}

\newcommand{\ddt}{\frac{d}{dt}}

\newcommand{\dds}{\frac{d}{ds}}

\renewcommand{\P}{\mathbb{P}}
\newcommand{\bA}{\mathbb{A}}

\newcommand\lie[1]{\mathfrak{#1}}
\renewcommand{\k}{\lie{k}}

\newcommand{\g}{\lie{g}}

\newcommand{\on}{\operatorname}
\newcommand{\ainfty}{{$A_\infty$\ }}

\newcommand{\Assoc}{\on{Assoc}}
\newcommand{\Cycl}{\on{Cycl}}

\newcommand{\Comm}{\on{Comm}}

\newcommand{\dual}{\vee}

\newcommand{\Ve}{\on{Vert}}
\newcommand{\Edge}{\on{Edge}}

\newcommand{\Ver}{\on{Vert}}

\newcommand{\Aut}{ \on{Aut} }

\newcommand{\Ad}{ \on{Ad} }

\newcommand{\Hom}{ \on{Hom}}

\newcommand{\Mult}{\on{Mult}}
\newcommand{\Vol}{\omega}

\newcommand{\dist}{\on{dist}}

\newcommand{\codim}{\on{codim}}

\newcommand{\ssm}{-}

\newcommand\dirac{/\kern-1.2ex\partial} 
\newcommand\qu{/\kern-.7ex/} 
\newcommand\lqu{\backslash \kern-.7ex \backslash} 
\newcommand\bs{\backslash}
\newcommand\dr{r_+ \kern-.7ex - \kern-.7ex r_-}
 
\usepackage{amsmath, amsfonts,amssymb, latexsym, mathrsfs}
\newtheorem{definition}[theorem]{Definition}
\newtheorem{proposition}[theorem]{Proposition}
\newtheorem{example}[theorem]{Example}

\renewcommand{\comment}[1]   {{\marginpar{*}\scriptsize{\ #1 \ }}}

\newcommand{\labell}\label

\renewcommand{\d}{{\on{d}}}
\newcommand{\ovl}{\overline}
\newcommand{\olp}{\ovl{\partial}}

\newcommand\Phinv{\Phi^{-1}}

\newcommand\eps{\epsilon}

\newcommand{\f}{\frac}
\newcommand{\lan}{\langle}
\newcommand{\ran}{\rangle}
\newcommand{\hh}{{\f{1}{2}}}

\newcommand{\ti}{\tilde}

\newcommand{\sss}{\on{ss}}

\renewcommand\AA{\mathcal{A}}

\newcommand\Map{\on{Map}}

\newcommand\ev{\on{ev}}

\newcommand\ul{\underline}
\newcommand\mO{\mathcal{O}}

\renewcommand\H{\mathcal{H}}

\newcommand\bra[1]{ < \kern-.7ex {#1} \kern-.7ex >} 
\newcommand\bdefn{\begin{definition}}
\newcommand\edefn{\end{definition}}
\newcommand\bea{\begin{eqnarray*}}
\newcommand\eea{\end{eqnarray*}}
\newcommand\bcv{\left[ \begin{array}{r} }
\newcommand\ecv{\end{array} \right] }

\newcommand\bma{\left[ \begin{array} }
\newcommand\ema{\end{array} \right]}
\newcommand\ben{\begin{enumerate}}
\newcommand\een{\end{enumerate}}
\newcommand\beq{\begin{equation}}
\newcommand\eeq{\end{equation}}
\newcommand\bex{\begin{example}}
\newcommand\bsj{\left\{ \begin{array}{rrr} }
\newcommand\esj{\end{array} \right\}}

\newcommand\Conf{\on{Conf}}

\newcommand\eex{\end{example}}

\newcommand\sx{*\kern-.5ex_X}

\newcommand{\fr}{{\on{fr}}}

\newcommand{\cA}{{\mathcal{A}}}

\def\mathunderaccent#1{\let\theaccent#1\mathpalette\putaccentunder}
\def\putaccentunder#1#2{\oalign{$#1#2$\crcr\hidewidth \vbox
to.2ex{\hbox{$#1\theaccent{}$}\vss}\hidewidth}}

\begin{document}

\title[Quantum Kirwan morphism I]{Quantum Kirwan morphism and Gromov-Witten invariants of
  quotients I}

\authors{Chris T. Woodward\thanks{Partially supported by NSF
 grant DMS0904358 and the Simons Center for Geometry and Physics}
\address Department of Mathematics \\ 
Rutgers University \\ 110 Frelinghuysen Road \\ Piscataway, NJ 08854-8019,
U.S.A. 
  \email ctw@math.rutgers.edu
}


\maketitle

\begin{abstract}  
This is the first in a sequence of papers in which we construct a
quantum version of the Kirwan map from the equivariant quantum
cohomology $QH_G(X)$ of a smooth polarized complex projective variety
$X$ with the action of a connected complex reductive group $G$ to the
orbifold quantum cohomology $QH(X \qu G)$ of its geometric invariant
theory quotient $X \qu G$, and prove that it intertwines the genus
zero gauged Gromov-Witten potential of $X$ with the genus zero
Gromov-Witten graph potential of $X \qu G$.  In this part we introduce
the moduli spaces used in the construction of the quantum Kirwan
morphism.
\end{abstract} 

\tableofcontents

\section{Introduction}

This is the first in a sequence of papers in which we construct a
quantum version of the morphism studied by Kirwan \cite{ki:coh}, which
maps the equivariant cohomology of a Hamiltonian group action to the
cohomology of the symplectic quotient.  The existence of a quantum
version was suggested by Salamon and Ziltener \cite{zilt:phd},
\cite{zilt:qk}.  Here we work under the assumption that the target is
a smooth projectively-embedded variety with a connected reductive
group action such that the stable locus is equal to the semistable
locus; this allows us to use the virtual fundamental cycle machinery
of Behrend-Fantechi \cite{bf:in}.  We prove that the quantum Kirwan
map intertwines the Gromov-Witten graph potential of the quotient with
the {\em gauged Gromov-Witten} potential of the action in the large
area limit.  In physics language, the quantum Kirwan map relates
correlators of a (possibly non-linear, non-abelian) gauged sigma model
with those of the sigma model of the quotient.  As such, the results
overlap with those of Givental \cite{gi:eq}, Lian-Liu-Yau
\cite{lly:mp1}, Iritani \cite{iri:gmt} and others.  The connection to
mirror symmetry is explained in the paper of Hori-Vafa \cite{ho:mi}:
because mirror symmetry for vector spaces is rather trivial, the
non-trivial change of coordinates arises when passing from a gauged
linear sigma model to the sigma model for the quotient.  Since the
quantum Kirwan map is defined geometrically, it can be rather
difficult to compute and the algebraic approach in \cite{gi:eq},
\cite{lly:mp1} is more effective in cases where it applies.  The
geometric approach pursued here has the advantage that are no
semipositivity assumptions on $X \qu G$ or abelian-ness assumptions on
the group $G$.  Also, $X$ can be a projective variety rather than a
vector space.  At the time that we started the project, there were
rather few papers about these situations; however in the meantime
papers such as Ciocan-Fontanine-Kim \cite{ciocan:bigI},
\cite{ciocan:genuszerowall} and Coates-Corti-Iritani-Tseng
\cite{coates:mirrorstacks} substantially extend the hypotheses of the
previous theorems.  However, the approaches are still quite different:
In \cite{coates:mirrorstacks}, the fundamental solution to the quantum
differential equation is expressed using a characterization of
Givental's Lagrangian cone for toric stacks, while in
\cite{ciocan:bigI}, \cite{ciocan:genuszerowall} the relationship to
the Gromov-Witten invariants of the quotient is given by a
wall-crossing formula rather than an adiabatic limit, and the proof
involves localization using an auxiliary group action.

To explain the gauged Gromov-Witten potential, recall that the {\em
  equivariant cohomology} of a $G$-space $X$ is the cohomology of the
homotopy quotient $X_G = EG \times_G X$ where $EG \to BG$ is a
universal $G$-bundle.  Equivariant quantum cohomology should count
maps $u: C \to X_G$ where $C$ is a curve equipped with some additional
data and $u$ is a holomorphic map.  Any such map can be viewed as a
map to $BG = EG/G$ together with a lift to $X_G$.  Holomorphic maps to
$BG$ correspond to holomorphic $G$-bundles, and so a holomorphic map
to $X_G$ is given by a holomorphic $G$-bundle $P \to C$ together with
a holomorphic section of the associated $X$ bundle $u: C \to P(X) := P
\times_G X$.  Givental \cite{gi:eq} had earlier introduced an {\em
  equivariant Gromov-Witten theory}, base on equivariant counts of
maps from a curve to $X$.  These counts give rise to a family of
products on the {\em equivariant quantum cohomology} $ QH_G(X) =
H_G(X) \otimes \Lambda^G_X $ where convergence issues are solved by
the introduction of the {\em Novikov field} $\Lambda^G_X \subset
\Hom(H^G_2(X,\Z),\Q) $.  In the language of maps to the classifying
space $X_G$, Givental's equivariant Gromov-Witten theory corresponds
to counting maps from $C$ to $X_G$ whose image lies in a single fiber
of the projection $X_G \to BG$, that is, such that the $G$-bundle is
trivial.  In order to distinguish the theory here from that of
Givental, we will call the theory with non-trivial bundles {\em gauged
  Gromov-Witten theory}, and call the map $u: C \to X_G$ a {\em
  gauged} map to $X$.

Gauged Gromov-Witten invariants should be defined as integrals over
moduli spaces of gauged maps.  In order to obtain proper moduli spaces
one needs to impose a {\em stability} or {\em moment map} condition as
well as compactify the moduli space by, for example, allowing {\em
  bubbling}.  Mundet and Salamon's work, see \cite{ci:symvortex},
provides a symplectic approach to the moduli spaces of gauged maps and
construction of invariants in the case that the target is a vector
space.  Mundet's thesis \cite{mund:corr} connects the symplectic
approach to an algebraic stability condition, which combines the
Ramanathan stability condition for principal bundles with the
Hilbert-Mumford stability for the action.  Schmitt
\cite{schmitt:univ}, \cite{schmitt:git} had earlier constructed a
Grothendieck-style compactification in the case that the target is a
smooth projective variety, while in the symplectic setting Mundet
\cite{mun:ham} and Ott \cite{ott:remov} show that the connected
components of the moduli spaces of semistable gauged maps have a
Kontsevich-style compactification.

In general, one needs virtual fundamental cycles to define integration
over the moduli spaces.  Here we restrict to the case that the target
$X$ is a projective $G$-variety and note that a pair $(P \to C, u: C
\to P(X))$ is by definition a morphism from $C$ to the {\em quotient
  stack} $X/G$ introduced by Deligne-Mumford \cite{dm:irr}.  We then
use the theory of virtual fundamental classes developed by
Behrend-Fantechi \cite{bf:in}, based on earlier work of Li-Tian
\cite{litian:vir}, to define gauged Gromov-Witten invariants.  The
symplectic geometry is then only used as motivation, and to show that
the Deligne-Mumford stacks that arise are proper.  Of course it is
desirable to have semistable reduction theorems to show properness but
from the algebraic perspective even the stability conditions are
somewhat obscure and we prefer the symplectic route to properness.  A
good example is the moduli space of semistable bundles on a curve,
where properness is immediate from the Narasimhan-Seshadri description
as unitary representations of the fundamental group but semistable
reduction is more involved. 

The theory of gauged Gromov-Witten invariants in genus zero fits into
an algebraic formalism that is a ``complexification'' of the theory of
{\em homotopy associative}, or \ainfty, algebras introduced by
Stasheff \cite{st:hs}.  In the \ainfty story, which roughly
corresponds to ``open strings'' in the mathematical physics language,
one has notions of \ainfty algebras, \ainfty morphisms, and \ainfty
traces.  These notions are associated with different polytopes called
the associahedra, multiplihedra, and cyclohedra respectively.  The
complexifications of these spaces are the Grothendieck-Knudsen space
$\ovl{\M}_{0,n}$ of stable $n$-marked genus $0$ curves, Ziltener
compactification $\ovl{\M}_{n,1}(\bA)$ of $n$-marked $1$-scaled affine
lines, and the Fulton-MacPherson space $\ovl{\M}_n(\P) :=
\ovl{\M}_{0,n}(\P,[\P])$ of stable $n$-marked, genus $0$ maps of class
$[\P]$ to the projective line $\P$ respectively.  The first space
leads to the notion of {\em genus zero cohomological field theory
  (CohFT)}, in particular, a {\em CohFT algebra} given by the
invariants with multiple incoming markings and a single outgoing
marking.  The second space is associated to the notion of {\em
  morphism of CohFT algebras}, and the third to the notion of {\em
  trace on a CohFT algebra}.  The following is proved in
Gonzalez-Woodward \cite{cross}, and is reviewed in \cite[Theorem
  7.20]{qk3}.

\begin{theorem} [Gauged Gromov-Witten invariants]  \cite{cross}
\label{theorem1}
Let $X$ be a smooth polarized projective $G$-variety and $C$ a smooth
connected projective curve.  Suppose that every semistable gauged map
is stable; then the category of stable gauged maps is a proper
Deligne-Mumford stack equipped with a perfect obstruction theory.  The
gauged Gromov-Witten invariants $ (\tau_X^{G,n})_{n \ge 0}: QH_G(X)^n
\times H_n(\ovl{\M}_n(\P)) \to \Lambda_X^G $ define a trace on
$QH_G(X)$.
\end{theorem} 
The gauged Gromov-Witten invariants are also defined for polarized
quasiprojective varieties under suitable properness assumptions for
the moduli stacks of gauged maps, for example, if the polarization
corresponds to an equivariant symplectic form with proper moment map
convex at infinity.  In particular, gauged Gromov-Witten invariants
are defined for a vector space $X$ equipped with the action of a torus
$G$ such that the weights are contained in an open half-space in the
space of rational weights; this condition implies in particular that
the quotient $X \qu G$ is proper under the stable=semistable
condition.

One can organize the gauged Gromov-Witten invariants into a formal
{\em gauged Gromov-Witten potential}
$$ \tau^G_X: QH_G(X) \to \Q, \quad \alpha \mapsto \sum_{n \ge 0}
\tau_X^{G,n}(\alpha,\ldots, \alpha;1) /n!  .$$
Here ``formal'' means that the sum may not converge and we treat
$\tau^G_X$ as a Taylor series.  The splitting axiom implies that the
bilinear form constructed from the second derivatives of the potential
is compatible with the quantum product $\star_\alpha$ on $T_\alpha
QH_G(X)$ in the usual sense.  However this bilinear form will usually
be degenerate and so will not define a family of Frobenius algebra
structures.  It is convenient to throw into the definition of the
trace certain {\em Liouville classes} on the moduli space of gauged
maps, in which case, if $G$ is trivial, $\tau^G_X$ becomes equivalent
to the {\em graph potential} considered in Givental \cite{gi:eq}.

From the definition of Mundet stability one expects the gauged
Gromov-Witten invariants to be related in the limit that the
equivariant symplectic class $[\omega_{X,G}]$ approaches infinity to
the Gromov-Witten invariants of the quotient $X \qu G$, or more
precisely, to the genus zero {\em graph potential} for the geometric
invariant theory quotient $X \qu G$
$$ \tau_{X \qu G}: QH(X \qu G) \to \Lambda_{X \qu G}.$$
For our purposes, it is more natural to work over the larger Novikov
field $\Lambda_X^G$.  

However one sees quickly that the two potentials above cannot be
equal, most obviously because they have different domains.
Gaio-Salamon \cite{ga:gw} showed that the limiting process in which
the area of the domain is taken to infinity involves various kinds of
bubbling which one hopes to incorporate into a description of the
relationship.  Salamon and Ziltener \cite{zilt:phd} suggested that a
map $\kappa^G_X: QH_G(X) \to QH(X \qu G)$ might by defined by counting
the ``affine vortices'' that arise in the large area limit.  By a
Hitchin-Kobayashi correspondence with Venugopalan \cite{venu:heat},
\cite{venuwood:class} these maps correspond to the following algebraic
objects: 

\begin{definition}  \label{affinegauged} {\rm (Affine gauged maps)}  
In the case that $X \qu G$ is a free quotient, an {\em affine gauged
  map} to $X$ consists of a tuple
%
$$(p: P \to \P, \lambda : \P \to T^\dual \P \otimes \mO(2\infty), u:
\P \to P(X),\ul{z} \in (\P)^n)$$
consisting of 
\begin{enumerate} 
\item (Scaling form) a meromorphic one-form $\lambda$ on $\P$ with
  only a double pole at $\infty \in \P$ (hence inducing an affine
  structure on $\P - \{ \infty \}$);
\item (Morphism to the quotient stack, stable at infinity) a
  morphism $\P \to X/G$, consisting of a $G$-bundle $p: P \to \P$
  and a section $u: \P \to P(X)$ such that $u(\infty)$ is contained
  in the open subvariety $P(X^{\sss})$ associated with the semistable
  locus $X^{\sss} \subset X$; and
\item (Markings) an $n$-tuple of distinct points $\ul{z} = (z_1,\ldots,z_n) \in (\P)^n$. 
\end{enumerate} 
\end{definition} \vskip .1in 

In the case that the action of $G$ on the semistable locus of $X$ is
only locally free we also allow a stacky structure at infinity.  That
is, for $r > 0$ an integer we denote by $\mu_r$ the group of $r$-th
complex roots of unity and by 
$$\P[1,r] := ( \C^2 - \{ 0 \}) / \C^\times $$
the weighted projective line with a $\mu_r$-singularity at $\infty$,
where $\C^\times$ acts on $\C^2$ with weights $1,r$.  The morphism $u$
is then required to be a representable morphism from some $\P[1,r]$ to
$X/G$.

\begin{theorem}
 [Quantum Kirwan Morphism] \label{large} Suppose that $X$ is a
 projective $G$-variety equipped with a polarization such that every
 semistable point is stable.  Integrating over a compactified stack of
 affine gauged maps defines a morphism of CohFT algebras
$$\kappa_X^G: QH_G(X) \to QH(X \qu G) .$$
\end{theorem} 
\noindent By definition a morphism of CohFT algebras consists of a sequence of maps
$$ \kappa_X^{G,n}: QH_G(X)^{n} \times H(\ovl{\M}_{n,1}(\bA))
\to QH(X \qu G), \quad n \ge 0 $$
satisfying a splitting axiom that guarantees that the formal map
$$ {\kappa_X^G}: QH_G(X) \to QH(X \qu G), \quad \alpha \mapsto \sum_{n
  \ge 0} {\kappa}_X^{G,n}(\alpha,\ldots, \alpha;1) /n! $$
induces a $\star$-homomorphism on each tangent space.  In the case
that the {\em curvature} $\kappa_X^{G,0}$ of the quantum Kirwan
morphism vanishes, one obtains in particular a morphism of small
quantum cohomology rings
$$ \kappa_{X}^{G,1} : T_0 QH_G(X) \to T_0 QH(X \qu G) .$$
In this sense, morphisms of CohFT algebras can be considered as
non-linear generalizations of algebra homomorphisms.  

More precisely the quantum Kirwan morphism is defined by virtual
integration over a compactification $\ovl{\M}^G_{n,1}(\bA,X)$ of the
moduli stack of affine gauged maps equipped with evaluation and
forgetful maps
$$ \ev \times \ev_\infty: \ovl{\M}^G_{n,1}(\bA,X) \to (X/G)^n \times \ovl{I}_{X \qu G},
\quad f: \ovl{\M}^G_{n,1}(\bA,X) \to \ovl{\M}_{n,1}(\bA) $$
where $\ovl{I}_{X \qu G}$ is the rigidified inertia stack appearing in
orbifold Gromov-Witten theory.  Properness of this moduli space
follows from compactness results of Ziltener \cite{zilt:phd},
\cite{zilt:qk} (for affine vortices) and a Hitchin-Kobayashi
correspondence for affine vortices due to Venugopalan and the author
\cite{venuwood:class}.  This space again has a perfect relative
obstruction theory, and pull-push using the virtual fundamental class
gives rise to the maps $\kappa^G_{X,n}$.  From the physics point of
view, Witten \cite{wi:ph} explained the relationship between
correlators in gauged sigma models and sigma models with target the
symplectic quotient as a kind of ``renormalization'' given by
``counting pointlike instantons''.  The ``quantum Kirwan map'' is a
precise mathematical meaning for part of this statement for arbitrary
gauged (possibly non-linear) sigma models.

Despite the complicated-looking definition, the stack of affine gauged
maps is easy to understand in simple cases.  

\begin{example} {\rm (Toric orbifolds)}  Let 
$X$ a vector space and $G$ a torus with Lie algebra $\g$ acting with
  weights $\mu_1,\ldots,\mu_k \in \g^\dual$ contained in a half-space
  a real form $\g^\dual_\R$.  Assume that stable=semistable, so that
  the quotient $X \qu G$ is a proper toric Deligne-Mumford stack.
  Then $\M_{1,1}^G(\bA,X)$ is isomorphic to the stack of morphisms $u$
  from $\bA$ to $X$ (that is, $X$-valued polynomials in a single
  variable) satisfying the following conditions for any $d \in \g_\Q$:
\begin{enumerate} 
\item {\rm (Degree Restriction)} the $j$-th component $u_j$ of $u$ has
  degree at most $(d,\mu_j), j = 1,\ldots, k$, that is,
$$ u_j(z) = a_{j,0} + a_{j,1}z + \ldots + a_{j,\lfloor (d,\mu_j)
    \rfloor } z^{\lfloor (d,\mu_j) \rfloor } $$
for some constants $a_{j,k} \in \C$.  If $(d,\mu_j)$ is not an integer
let $a_{j,(d,\mu_j)} = 0$.
\item {\rm (Stability condition)} The collection of leading order
  coefficients
$$u(\infty) := (a_{j,(d,\mu_j)})_{j=1}^k \in X^{\exp(d)}$$
lies in the semistable locus $X^{\sss}$ of $X$ and so defines a point
in the substack of the inertia stack given by $\{ \exp(d) \times
X^{\exp(d),\sss} \} / G \subset I_{X \qu G}$.
\end{enumerate} 

If $X \qu G$ is Fano with minimal Chern number at least two then
$\kappa_X^G(0) = 0 $ and $D_0 \kappa_X^G$ is an algebra homomorphism
of small quantum cohomologies.  Thus knowing the map $D_0 \kappa_X^G$
allows to give a presentation of the small quantum cohomology of $X
\qu G$.
\begin{enumerate}
\item {\rm (Projective space)} If $G = \C^\times$ acts by scalar
  multiplication on $X = \C^k$ then there is a unique morphism $u: \bA
  \to X$ such that all components have degree $1$,
$$ u(z) = (a_{1,0} + a_{1,1} z, \ldots, a_{k,0} + a_{k,1} z) $$
with marking at $z_1 = 0$, given limit at infinity
$$ \ev_\infty u = [ a_{1,1},\ldots, a_{k,1} ] \in \P^{k-1} = X \qu
G $$
and vanishing value 
$$ \ev_0 u = u(0) = (a_{1,0},\ldots,a_{k,0}) \in \C^k = X $$
at $z_1 = 0 \in \C \cong \bA$.  Interpreting $\xi^k$ as the Euler
class of $\C^k$ this implies that
$$D_0\kappa_X^G(\xi^k) = q $$ 
where $\xi$ is the generator of $QH_G(X)$ and $k = \dim(X)$.  See
Lemma 8.8 of \cite{qk3} for a justification that the higher
codimension strata may be ignored.  This gives the standard
presentation
$$QH(X \qu G) = \Lambda_X^G[\xi]/(\xi^k - q)$$ 
of quantum cohomology of projective space.
\item {\rm (The teardrop orbifold, a weighted projective line)}
  Suppose that $G = \C^\times$ acts on $X = \C^2$ with weights $1,2$,
  so that $X \qu G = \P[1,2]$ is the teardrop orbifold.  Morphisms of
  class $d \in \Q \cong H_2^G(X,\Q)$ exist only if $d \in \Z/2$, and
  if so are given by pairs of polynomials
%
$$ u(z) = (a_{1,0} + \ldots + a_{1,\lfloor d \rfloor} z^{\lfloor d
    \rfloor}, a_{2,0} + \ldots + a_{2,2d} z^{2d}) .$$
 In zero degree, $\ovl{M}_{1,1}^G(\bA,X,0) = X \qu G$ implies that $D_0
 \kappa_X^G(1) = 1$.  We may interpret $(2\xi)^{2d} \xi^{\lfloor d
   \rfloor}$ as the Euler class of the vector bundles corresponding to
 the derivatives of $u$ at $0$.  Integrating the Euler class amounts
 to counting such maps whose derivatives $j!  a_{1,j}, j < d, \ j!
 a_{2,j}, j < 2d$ are zero at the marking and, if $d$ is even, have
 semistable leading order terms
$$ \ev_\infty u = [ a_{1,d}, a_{2,2d}] \in \P[1,2] $$
see Lemma 8.8 of \cite{qk3}.  There is a unique such $[u]$ for a given
$\ev_\infty u$.  For $d = 1/2$ this implies
$$D_0 \kappa_X^G(\xi^2) =  1_{\Z_2}/2 $$ 
half the generator $1_{\Z_2}$ of the twisted sector.  For $d = 1$ this
implies that
$$D_0 \kappa_X^G(\xi^3) = q/4 .$$  
Hence the presentation of the quantum cohomology of the teardrop
orbifold 
$$QH( \P[1,2])= \Lambda_X^G[\xi]/(q - 4\xi^3) $$ 
which is a special case of Coates-Lee-Corti-Tseng \cite{coates:wps}.
\end{enumerate} 
See Example 5.32 for more details.  A similar strategy gives relations
for the quantum cohomology of any proper toric Deligne-Mumford stack
with projective moduli space.  In addition, the setup also applies to
quotients by non-abelian groups such as quiver varieties.
\end{example}

One of the main results of this paper is that in the large area limit
the graph potentials are naturally related via the quantum Kirwan
morphism in the following sense.  The quantum Kirwan map has an
$S^1$-equivariant extension
$\kappa_{X,G}: QH_G(X) \to QH(X \qu G)[[\hbar]]$ (which for any fixed
power of $q$, involves only finitely many powers of $\hbar$).  The
trace $\tau_{X \qu G}$ has a natural $S^1$-equivariant extension given
by interpreting the $\hbar$ as the first Chern class of the tangent
line at the marking (that is, the opposite of the $\psi$-class).  We
fix a symplectic class $[\omega_{X,G}] \in H^2_G(X)$ and consider the
stability conditions corresponding to the classes
$\rho [\omega_{X,G}]$ as $\rho \in (0,\infty)$ varies.

\begin{theorem}[Adiabatic Limit Theorem]
\label{largearea} 
Suppose that $C$ is a smooth projective curve and $X$ a polarized
projective $G$-variety such that stable=semistable for gauged maps
from $C$ to $X$ of sufficiently large $\rho$.  Then
\begin{equation} \label{largerel}
 {\tau}_{X \qu G} \circ {\kappa_X^G} = \lim_{\rho \to \infty}
 {\tau}_X^G. \end{equation}
\end{theorem} 

\noindent In other words, the diagram 
$$\begin{diagram} \node{QH_G(X)} \arrow{se,b}{\tau_X^G}
  \arrow[2]{e,t}{{\kappa}_X^G} \node{} \node{QH_{S^1}( X \qu G)}
  \arrow{sw,b}{\tau_{X \qu G}} \\ \node{} \node{\Lambda_X^G} \node{}
   \end{diagram}$$
commutes in the limit $\rho \to \infty$.\footnote{The $S^1$-subscript
  was omitted in the published version.}   The equality, or
commutativity of the diagram, holds in the space of distributions in
$q$, in other words, for each power of $q$ separately.  The result has
some analogies with the ``quantization commutes with reduction''
theorem of Guillemin-Sternberg \cite{gu:ge}.  However, in the
intervening years the use of ``quantum'' has changed so that now it
often refers to holomorphic curves.  The terminology ``adiabatic''
arises from the fact that stable gauged maps correspond to minima of
an energy function depending on $\rho$, so the theorem relates minima
in the limit $\rho \to \infty$.

One is often interested not in the graph Gromov-Witten invariants but
rather in the {\em localized graph potential} that arises as the fixed
point contributions for a circle acting on the domain.  As in the work
of Givental \cite{gi:eq} this localization in the case of
Gromov-Witten invariants gives rise to a solution $\tau_{X \qu G,\pm}$
for the quantum differential equation for $X \qu G$: for $\alpha \in
QH(X \qu G), \nu \in T_\alpha QH(X \qu G)$ after localization of the
equivariant parameter $\hbar$: A formal map
$$ \tau_{X \qu G,\pm}: QH(X \qu G) \to QH(X \qu G)[[ \hbar^{-1}]],
\quad (\pm \hbar) \partial_\nu \tau_{X \qu G,\pm} (\alpha) = \nu
\star_\alpha \tau_{X \qu G,\pm}(\alpha) .$$
The components of $\tau_{X \qu G,\pm}$ satisfy a version of the
Picard-Fuchs equations which play an important role in mirror symmetry
\cite{gi:eq}.  There are similar gauged versions (again formal)
$$ \tau_{X,\pm}^G : QH_G(X) \to QH(X \qu G)[[\hbar^{-1}]] $$
which capture the contributions from $0,\infty \in \P$ to the
localization formula applied to the {\em gauged} graph potential.
 The
factorization of the graph potentials generalizes to the gauged
setting and we prove a localized adiabatic limit Theorem
\ref{Jresults} below for the contributions to the fixed point formula.

\begin{theorem} 
[Localized adiabatic limit theorem]
\label{Jresults} \label{Jlargerel} \label{Jlargearea}
Under the assumptions of Theorem \ref{largerel},
$ \tau_{X \qu G,\pm} \circ \kappa_X^G = \tau_{X,\pm}^G .$
\end{theorem} 

In other words, after composition with the quantum Kirwan map the
localized graph potential is equal to the localized gauged potential.
The result also holds in the ``twisted case'', that is, after
inserting the Euler class of the index of an equivariant bundle on the
target, which in good cases describes the localized graph potentials
of complete intersections.  In this way one obtains formula similar to
the ``mirror formulas'' of Givental \cite{gi:eq}, Lian-Liu-Yau
\cite{lly:mp1}, Iritani \cite{iri:gmt} and others for localized graph
potentials \cite{gi:eq} in the toric case, but now for arbitrary
geometric invariant theory quotients.  However, the approach here is
different from that of \cite{gi:eq}, \cite{lly:mp1} etc. in that the
``mirror map'' is expressed as integrals over moduli spaces, while the
approach of \cite{gi:eq}, \cite{lly:mp1} etc.  solves for the ``mirror
map'' as the solution to an algebraic equation.  Recently other
authors have given geometric interpretations of the mirror map
\cite{lau:open}, \cite{ciocan:genuszerowall}.  

Various applications have been developed jointly with E. Gonzalez.  We
were rather surprised to discover that many of the ``standard
formulas'' from classical equivariant symplectic geometry generalize
to the quantum case by substituting the quantum Kirwan morphism for
the classical Kirwan map; this is rather unexpected since all of these
formulas involve functoriality of cohomology in some way which is
generally lacking in the quantum setting.  For example, there is (i) a
wall-crossing ``quantum Kalkman'' formula for Gromov-Witten invariants
under variation of quotient, including invariance in the case of
crepant flops (ii) an abelianization ``quantum Martin'' formula
relating Gromov-Witten invariants of quotients by connected reductive
groups and their maximal tori, first suggested by Hori-Vafa
\cite[Appendix]{ho:mi} and Bertram-Ciocan-Fontanine-Kim \cite{be:qu}
and (iii) a quantum version of Witten's non-abelian localization
principle, relating the equivariant quantum cohomology correlators for
$X$ with the quantum cohomology correlators for $X \qu G$.

We remark that Ciocan-Fontanine-Kim-Maulik \cite{cf:st} have
introduced a different notion of gauged Gromov-Witten invariants of
quotients of affine varieties which works in any genus.  A symplectic
interpretation of these invariants has recently been given by
Venugopalan \cite{venu:cyl}.  From this point of view it seems that
the Ciocan-Fontanine-Kim-Maulik invariants are a higher genus
generalization of the ``large area chamber'' gauged Gromov-Witten
invariants investigated here.  See \cite{ciocan:hg},
\cite{ciocan:bigI} for more recent work.

\section{Traces and morphisms of cohomological field theory
  algebras}

To state the main result precisely, we explain what it means to have a
``commutative diagram'' of cohomological field theories.  In this
section we describe the moduli spaces of stable curves (complexified
associahedron in genus zero), stable parametrized curves (complexified
cyclohedron) and stable affine scaled curves (complexified
multiplihedron), which lead to the notion of CohFT algebra, trace on a
CohFT algebra, and morphism of CohFT algebras respectively, in analogy
with the theory of \ainfty spaces, morphisms, and traces.  Then we
introduce notions of compositions of morphisms and traces, or
morphisms of CohFT algebras, which are analogous to the composition of
the corresponding \ainfty notions.  This makes CohFT algebras into a
kind of $\infty$-category.  Notably, we do not have a version of
complexified multiplihedron for higher genus curves, which is why the
theory here is restricted to genus zero.
 
\subsection{Complexified associahedron and CohFT algebras} 

Moduli spaces of stable curves were introduced by Mayer and Mumford
\cite{woodshole}, and further studied in Deligne-Mumford
\cite{dm:irr}.  Moduli of stable marked curves were studied by
Grothendieck in 1968 and later by Knudsen \cite{kn:proj2}.  In this
section we describe these compactifications and the notion of {\em
  cohomological field theory} introduced by Kontsevich-Manin, see
\cite{man:fro}.  We remark that since the notion of stack is not
introduced until \cite[Section 4]{qk2}, we avoid it until then and
adopt the point of view that the moduli spaces are just topological
spaces.  First we describe stable curves.

\begin{definition}  
 \label{stablecurves0}
Let $n \ge 0$ be an integer.
\begin{enumerate} 
\item {\rm (Nodal curves)} An {\em $n$-marked nodal curve} consists of
  a projective nodal curve $C$ with an $n$-tuple of distinct,
  non-singular points $ \ul{z} = (z_1,\ldots, z_n) \in C^n$.  An {\em
    isomorphism} of $n$-marked nodal curves $(C,\ul{z}), (C',\ul{z}')$
  is an isomorphism $\phi: C \to C'$ such that $\phi(z_i) = z_i'$ for
  $i = 1,\ldots, n$.
\item 
{\rm (Stable curves)}
 A nodal $n$-marked curve $C = (C,z)$ is {\em stable}
iff $C$ has finite automorphism group.  That is, each genus zero
component has at least three {\em special points} (nodes or markings)
and each genus one component has at least one special point.  Note
that we do not require $C$ to be connected. 
\item {\rm (Modular graphs)} The combinatorial type $\Gamma$ of a
  stable curve is a {\em modular graph}: 
\begin{enumerate} 
\item {\rm (Graph)} an unoriented graph $\Gamma =
  (\Ve(\Gamma),\Edge(\Gamma))$ whose vertices correspond to
  irreducible components of $C$, finite edges $\Edge_{<
    \infty}(\Gamma)$ to nodes, and semi-infinite edges
  $\Edge_\infty(\Gamma)$ to markings, equipped with a 
\item {\rm (Genus function)} $\Z_{\ge 0}$-valued function $g:
  \Ve(\Gamma) \to \Z_{\ge 0}$ recording the genus of each irreducible
  component of $C$ and 
\item {\rm (Labelling of semi-infinite edges)} a bijection $l:
  \Edge_\infty(\Gamma) \to \{ 1, \ldots n \}$ of the semi-infinite
  edges with labels $1,\ldots, n$. 
\end{enumerate} 
A modular graph is {\em stable} if it corresponds to a stable curve.
That is, each vertex with label $0$ resp. $1$ has valence at least $3$
resp $1$.
\end{enumerate} 
\end{definition} \vskip .1in

Any stable curve $C$ has a {\em universal deformation} given by a
family of curves $\pi: C_S \to S$ over a parameter space $S$ uniquely
defined up to isomorphism, see for example \cite[p. 184]{ar:alg2}.
Let $\ovl{M}_{g,n}$ denote the set of isomorphism classes of connected
genus $g$, $n$-marked stable curves.  More naturally one should
consider the moduli {\em stack} of stable curves but we put off
discussion of stacks to \cite[Section 4]{qk2}.  In the case of genus
zero curves the universal deformation, and topology on $\ovl{M}_{g,n}$,
have a simple description \cite[p. 184]{ar:alg2}, \cite[Appendix
  D]{ms:jh}: For any marking $z_i$ and irreducible component $C_j$ of
$C$ we denote by $z_i^j$ the node in $C_j$ connecting to the
irreducible component of $C$ containing $z_i$, or $z_i$ if $z_i$ is
contained in $C_j$.

\begin{definition} {\rm (Convergence of a sequence
of stable curves)} A sequence $[C_\nu]$ converges to $[C]$ in
  $\ovl{M}_{g,n}$ if $C_\nu$ is isomorphic to $\pi^{-1}(s_\nu)$ for a
  sequence $s_\nu$ converging to $s$ in the base $S$ of the universal
  deformation.  Explicitly, if $g = 0$, a sequence
  $[(C_\nu,z_{1,\nu},\ldots,z_{n,\nu})]$ with smooth domain $C_\nu$
  converges to $[(C,z_1,\ldots,z_n)]$ if there exists, for each
  irreducible component $C_j$ of $C$, a sequence of holomorphic
  isomorphisms $\phi_{j,\nu}: C_j \to C_\nu$ such that
\begin{enumerate}
\item {\rm (Limit of a marking)} for all $i,j$, $\lim_{\nu \to \infty}
  \phi_{j,\nu}^{-1}(z_{i,\nu}) = z_i^j$; and 
\item {\rm (Limit of a different parametrization)} for all $j \neq k$,
  $\lim_{\nu \to \infty} \phi_{j,\nu} \phi_{k,\nu}^{-1}$ has limit the
  constant map with value the node of $C_j$ connecting to $C_k$.
\end{enumerate} 
\end{definition} \vskip .1in 

With the topology induced by this notion of convergence,
$\ovl{M}_{g,n}$ is compact and Hausdorff (and in fact, a projective
variety \cite{dm:irr}.)  For any possible disconnected graph $\Gamma$
we denote by $M_{g,n,\Gamma}$ the space of isomorphism classes of
curves of combinatorial type $\Gamma$ with $n$ semiinfinite edges and
total genus $g$, and $\ovl{M}_{g,n,\Gamma}$ its closure.  If $\Gamma =
\Gamma_0 \sqcup \Gamma_1$ is a disjoint union then
$\ovl{M}_{g,n,\Gamma} = \ovl{M}_{g,n,\Gamma_0} \times
\ovl{M}_{g_1,n_1,\Gamma_1}$.  The moduli spaces of stable marked curves
$\ovl{M}_{n,\Gamma}$ satisfy a natural functoriality with respect to
morphisms of modular graphs $\Gamma$.

\begin{definition}  {\rm (Morphisms of modular graphs)} 
A {\em morphism} of modular graphs $\Upsilon: \Gamma \to \Gamma'$ is a
surjective morphism of the set of vertices $\Ve(\Gamma) \to
\Ve(\Gamma')$ obtained by combining the following: (these are called
   {\em extended isogenies} in Behrend-Manin \cite{bm:gw})
\begin{figure}[h]
\begin{picture}(0,0)%
\includegraphics{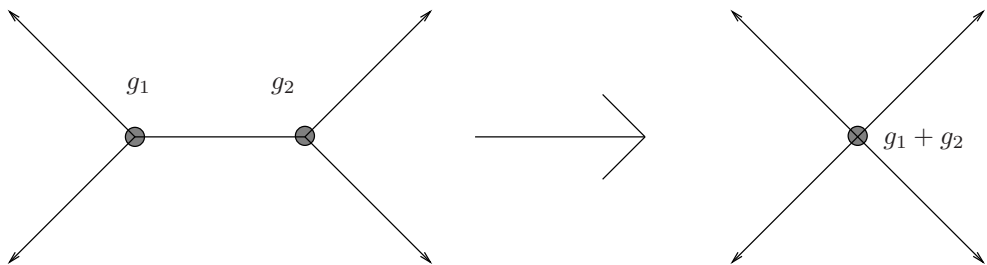}%
\end{picture}%
\setlength{\unitlength}{4144sp}%
\begingroup\makeatletter\ifx\SetFigFontNFSS\undefined%
\gdef\SetFigFontNFSS#1#2#3#4#5{%
  \reset@font\fontsize{#1}{#2pt}%
  \fontfamily{#3}\fontseries{#4}\fontshape{#5}%
  \selectfont}%
\fi\endgroup%
\begin{picture}(5874,1550)(2239,-2049)
\put(2963,-1020){\makebox(0,0)[lb]{{{{$g_1$}%
}}}}
\put(3828,-1020){\makebox(0,0)[lb]{{{{$g_2$}%
}}}}
\put(7491,-1350){\makebox(0,0)[lb]{{{{$g_1 + g_2$}%
}}}}
\end{picture}%

\caption{Collapsing an edge}
\label{collapsefig2}
\end{figure}

\begin{enumerate} 
\item $\Upsilon$ {\em collapses an edge} if the map on vertices
  $\Ve(\Upsilon): \Ve(\Gamma) \to \Ve(\Gamma')$ is a bijection except
  for a single vertex $v' \in \Ve(\Gamma')$ which has two pre-images
  connected by an edge in $\Edge(\Gamma)$, and $\Edge(\Gamma') \cong
  \Edge(\Gamma) - \{ e \}$.  The genus function on $\Gamma'$ is
  obtained by push-forward that is, $g(v') = \sum_{\Upsilon(v) = v'}
  g(v)$.

\begin{figure}[ht]
\begin{center} 
\begin{picture}(0,0)%
\includegraphics{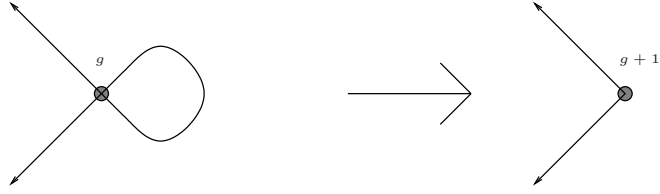}%
\end{picture}%
\setlength{\unitlength}{4144sp}%
\begingroup\makeatletter\ifx\SetFigFont\undefined%
\gdef\SetFigFont#1#2#3#4#5{%
  \reset@font\fontsize{#1}{#2pt}%
  \fontfamily{#3}\fontseries{#4}\fontshape{#5}%
  \selectfont}%
\fi\endgroup%
\begin{picture}(3748,1128)(2239,-1627)
\put(2767,-879){\makebox(0,0)[lb]{\smash{{\SetFigFont{5}{6.0}{\rmdefault}{\mddefault}{\updefault}$g$}%
}}}
\put(5900,-879){\makebox(0,0)[lb]{\smash{{\SetFigFont{5}{6.0}{\rmdefault}{\mddefault}{\updefault}$g+1$}%
}}}
\end{picture}%
\end{center} 
\caption{Collapsing a loop}
\label{collapseloop}
\end{figure}

\item $\Upsilon$ {\em collapses a loop} if the map on vertices is a
  bijection and $\Edge(\Gamma') \cong \Edge(\Gamma) - \{ e \}$ where $e$
  is an edge connecting a vertex $v$ to itself.  Then the genus
  functions on $\Gamma,\Gamma'$ are identical except at $v$ where
  $g'(v) = g(v) + 1$.

\item $\Upsilon$ {\em cuts an edge} $e \in \Edge(\Gamma)$ if the map
  on vertices is a bijection, but $\Edge(\Gamma') \cong \Edge(\Gamma) - \{
  e \} + \{ e_+,e_- \} $ where $e_\pm$ are semiinfinite edges attached
  to the vertices contained in $e$.  Any morphism $\Upsilon: \Gamma
  \to \Gamma'$ cutting an edge induces an isomorphism $
  \ovl{M}(\Upsilon): \ovl{M}_{g,n,\Gamma'} \to \ovl{M}_{g,n,\Gamma} $
  obtained by identifying the markings corresponding to the edges
  $e_\pm$ in $\Gamma'$.

\begin{figure}[ht]
\begin{picture}(0,0)%
\includegraphics{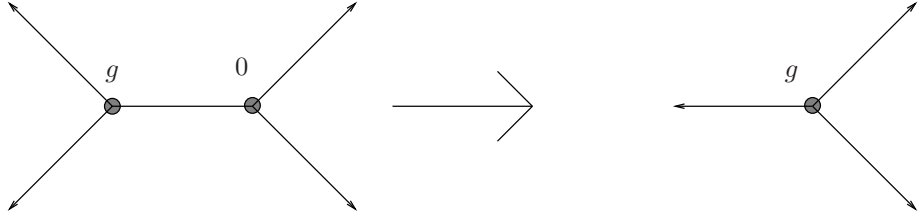}%
\end{picture}%
\setlength{\unitlength}{4144sp}%
\begingroup\makeatletter\ifx\SetFigFontNFSS\undefined%
\gdef\SetFigFontNFSS#1#2#3#4#5{%
  \reset@font\fontsize{#1}{#2pt}%
  \fontfamily{#3}\fontseries{#4}\fontshape{#5}%
  \selectfont}%
\fi\endgroup%
\begin{picture}(5469,1280)(2239,-1779)
\put(2837,-992){\makebox(0,0)[lb]{{{{$g$}%
}}}}
\put(3612,-951){\makebox(0,0)[lb]{{{{$0$}%
}}}}
\put(6900,-992){\makebox(0,0)[lb]{{{{$g$}%
}}}}
\end{picture}%
\caption{Forgetting a tail}
\label{forgetfig2}
\end{figure}

\item $\Upsilon$ {\em forgets a tail} (semiinfinite edge) $e \in
  \Edge^\infty(\Gamma)$ if the map on vertices is a bijection, but
  $\Edge(\Gamma')\cong \Edge(\Gamma) - \{ e \}$.  In this case there is a
  morphism $\ovl{M}(\Upsilon): \ovl{M}_{g,n,\Gamma} \to
  \ovl{M}_{g,n,\Gamma'}$ obtained by forgetting the corresponding
  marking and collapsing any unstable components.
\end{enumerate} 
\end{definition} \vskip .1in

\begin{proposition}  The boundary of $M_{g,n,\Gamma}$ in
$\ovl{M}_{g,n,\Gamma}$ is the union of spaces $M_{g,n,\Gamma'}$ such
  that $\Gamma$ is obtained from $\Gamma'$ by collapsing edges and
  forgetting loops.  The boundary of $\ovl{M}_{g,n}$ is a union of the
  following subspaces (which will be {\em divisors} with respect to
  the algebraic structure of the moduli space introduced later)
\begin{enumerate} 
\item {\rm (Non-separating node)} if $2g +n > 3$, a subspace
$$\iota_{g-1,n+2}: D_{g-1,n+2} \to \ovl{M}_{g,n}$$
equipped with an isomorphism 
$$ \varphi_{g-1,n+2}: D_{g-1,n+2} \to \ovl{M}_{g-1,n+2} .$$  
The inclusion is obtained by identifying the last two marked points.
\item {\rm (Separating node)} for each splitting $g = g_1 + g_2, \{ 1
  ,\ldots , n \} = I_1 \cup I_2$ with $2g_j + |I_j| \ge 3, j = 1,2$, a
  subspace
$$\iota_{g_1+g_2,I_1 \cup I_2}: D_{g_1+g_2,I_1 \cup I_2} \to
\ovl{M}_{g,n}$$
corresponding to the formation of a separating node,
splitting the surface into pieces of genus $g_1,g_2$ with markings
$I_1,I_2$, equipped with an isomorphism 
$$ \varphi_{g_1+g_2,I_1 \cup I_2}: D_{g_1 + g_2,I_1 \cup I_2} \to
\ovl{M}_{g_1,|I_1|+1} \times \ovl{M}_{g_2,|I_2|+1}$$ 
(except that on the level of orbifolds in the case $I_1 = I_2 =
\emptyset$ and $g_1 = g_2$ there is an additional automorphism
exchanging the components.)
\end{enumerate} 
\end{proposition}  

The pull-back $\iota_{g_1+g_2,I_1 \cup I_2}^* \beta$ of any class
$\beta \in H(\ovl{M}_{g,n})$ has a K\"unneth decomposition
\begin{equation} \label{kunneth1}
 \iota_{g_1+g_2,I_1 \cup I_2}^* \beta = 
\sum_{j \in J} \beta_{1,j} \otimes \beta_{2,j} \end{equation}
for some index set $J$ and classes $\beta_{k,j} \in
H(\ovl{M}_{g_k,|I_k|+1})$.  Each boundary divisor $D_{g-1,n+2}$ or
$D_{g_1+g_2,I_1 \cup I-2}$ has a homology class $\ovl{M}_{g,n}$, and
since $\ovl{M}_{g,n}$ is a rational homology manifold each of these
homology classes has a dual class $\gamma_{g-1,n+2}$
resp. $\gamma_{g_1+g_2,I_1 \cup I_2}$ in $H^2(\ovl{M}_{g,n},\Q)$.  For
the following, see Manin \cite{man:fro}.

\begin{definition}
 \label{cohft}
A {\em cohomological field theory} (CohFT) with values in a field
$\Lambda$ is a $\Z_2$-graded vector space $V$ equipped with a
symmetric non-degenerate bilinear form $ V \times V \to \Lambda$ and
collection of $S_n$-invariant (with Koszul signs) {\em correlators}
$$ V^{n} \times H(\ovl{M}_{g,n}) \to \Lambda, \ \ (\alpha,\beta)
\mapsto \langle \alpha ; \beta \rangle_{g,n}, \ g, n \ge 0 $$
%
where by convention $H(\ovl{M}_{g,n}) = \Lambda$ if $\ovl{M}_{g,n} =
\emptyset$, satisfying the following two splitting axioms:
\begin{enumerate} 
\item {\rm (Non-separating node)} 
if $g \ge 1$ then 
$$ \langle  \alpha ; \beta \cup \gamma_{g-1,n+2} \rangle_{g,n}  
= \sum_k \langle \alpha, \delta_k, \delta^k; \iota_{g-1,n+2}^* \beta
  \rangle_{g-1,n+2} $$
where $\delta_k,\delta^k$ are dual bases for $V$; 
\item {\rm (Separating node)} if $2g + n \ge 4$, $I_1 \cup I_2$ is a
  partition of $\{1 ,\ldots, n \}$, and $g = g_1 + g_2$ with $2g_i +
  |I_i| \ge 3$ for $i=1,2$ then
$$ \langle \alpha ; \beta \cup \gamma_{g_1+g_2,I_1 \cup I_2}
  \rangle_{g,n} = \sum_k \langle ( \alpha_i )_{i \in I_1}, \delta_k ;
  \cdot \rangle _{g_1,|I_1|+1} \langle (\alpha_i)_{i \in I_2}, \delta^k ; \cdot
  \rangle_{g_2,|I_2|+1} ( \iota_{g_1 + g_2,I_1 \cup I_2}^* \beta) $$
where the dots indicate insertion of the K\"unneth components of
$\iota_{g_1 + g_2,I_1 \cup I_2}^* \beta$, $\delta_k,\delta^k$ are dual
bases for $V$, and there is an additional factor of $2$ in the
exceptional case $g_1 = g_2$, $I_1 = I_2 = \emptyset$ arising from the
additional automorphism.
\end{enumerate} 
\end{definition} \vskip .1in

That is, if $\beta$ is as in \eqref{kunneth1} then
$$ \langle \alpha ; \beta \cup \gamma_{g_1+g_2,I_1 \cup I_2}
\rangle_{g,n} = \sum_{j \in J,k} \langle (\alpha_i)_{i \in I_1}, \delta_k ;
\beta_{1,j} \rangle _{g_1,|I_1|+1} \langle (\alpha_i)_{i \in I_2}, \delta^k ;
\beta_{2,j} \rangle _{g_2,|I_2|+1} .$$

The correlators of any cohomological field theory define a family of
associative algebra structures.  In the standard axiomatization these
are part of the associated {\em Frobenius manifold structure} on $V$
\cite{man:fro}.  However for the purposes of functoriality it is
helpful to keep the ``algebra'' and ``metric'' parts of this structure
separate, and instead we define the following:

\begin{definition}  A {\em CohFT algebra} consists of a $\Z_2$-graded vector
space $V$ and a collection of $S_n$-invariant (with Koszul signs)
multilinear maps
$$ \mu^n: V^n \times H(\ovl{M}_{0,n+1}) \to V, n \ge 2 $$
satisfying the splitting axiom for any subset $I \subset \{ 1,\ldots,
n \}$ of order at least two:
\begin{equation} \label{splitalg}
 \mu^n( \alpha_1,\ldots, \alpha_n; \gamma_{0,I \cup (I^c \cup \{ 0
   \})} \cup \beta) = \sum_{j} \mu^{n - |I| + 1}(\alpha_i, i \notin I,
 \mu^{|I|} (\alpha_i, i\in I; \beta_{1,j});
 \beta_{2,j}) \end{equation}
where $\iota_I^* \beta = \sum_j \beta_{1,j} \otimes \beta_{2,j}$ is the K\"unneth
decomposition of the restriction of $\beta$ as in \eqref{kunneth1}. 
\end{definition} \vskip .1in 

\begin{proposition}  Any CohFT gives rise to a CohFT algebra. 
\end{proposition} 

\begin{proof}  By restricting to genus zero, and using duality to put
one factor of $V$ on the right.  The splitting axiom \eqref{splitalg}
is a special case of the splitting axiom in Definition \ref{cohft}.
\end{proof} 

\begin{remark}
\begin{enumerate} 
\item Manin \cite{man:fro} terms a set of such maps a {\em
  $\Comm_\infty$-structure} on $V$.  However we avoid this terminology
  since it isn't clear what this is a homotopy version of.
\item The CohFT algebra structure does not require a metric, so a
  CohFT may be thought of roughly speaking as a CohFT algebra plus a
  metric.  We remark that a more natural definition of CohFT may be
  obtained by allowing incoming and outgoing markings.  However, we
  prefer to write the classes on the left to emphasize the analogy
  with \ainfty spaces.
\item The various relations in $\ovl{M}_{0,n+1}$ give rise to relations
  on the maps $\mu^n$. In particular the map $\mu^2: V \times V \to
  V$ is associative, by the splitting axiom and the relation 
$ [D_{0,
      \{0, 3 \} \cup \{ 1,2 \}}] = [D_{0, \{ 0,1 \} \cup \{ 2, 3\} }]
   \in H^2(\ovl{M}_{0,4})  .$
\item The notion of CohFT algebra is the ``complex analog'' of the
  notion of \ainfty algebra in the following imprecise sense.  Denote
  by $\ovl{M}_{0,n}^\R$ the moduli space of projective nodal curves $C$
  equipped with an anti-holomorphic involution fixing the markings,
  that is, the moduli space of {\em stable $n$-marked disks} where the
  markings are not necessarily in cyclic order.  The symmetric group
  $S_n$ acts canonically on $\ovl{M}_{0,n}^\R$ by permuting the
  markings.  The quotient of $\ovl{M}_{0,n}^\R$ by the action of $S_n$
  is homeomorphic to the $(n-1)$-st  associahedron $\Assoc_{n-1}$
  introduced in Stasheff \cite{st:hs}.  An {\em \ainfty space} is a
  space $X$ equipped with a collection of maps 
$X^n \times \Assoc_n
  \to X, n \ge 2$
satisfying a splitting axiom for the restriction to the boundary, and
{\em \ainfty algebras} arise as spaces of chains on \ainfty spaces.
To obtain the notion of CohFT algebra we replace $X$ by a vector space
and $\Assoc_n$ by the cohomology of its complexification.  One could
also imagine a cochain-level version but there are reasons to expect
that this gives nothing new, see Teleman \cite{tel:tft2}
\item The notion of a genus zero CohFT can be repackaged in terms of a
  non-linear structure called a {\em Frobenius manifold}, which is a
  non-linear generalization of the notion of {\em Frobenius algebra}
  (unital algebra with compatible metric.)  Any genus zero CohFT (with
  unit and grading) gives rise to a Frobenius manifold whose potential
  is
$$ f: V \to \Lambda, \quad f(v) = \sum_{n \ge 3} \frac{1}{n!} \lan
v,\ldots, v; 1 \ran_{0,n} .$$
The third derivatives of $f$ give rise to a family of algebra
structures
$$\star_v : T_v V^2 \to T_v V .$$
These give rise to a family of connections depending on a parameter
$\hbar$,
$$ \nabla_\hbar: \Omega^0(V,TV) \to \Omega^1(V,TV), \quad
\nabla_{\hbar,\xi} \sigma (v) = (\d \sigma(\xi))(v) -
(1/\hbar) \xi \star_{v} \sigma(v) .$$
The associativity of $\star$ translates into the flatness of
$\nabla_\hbar$, so that locally there exist sections $\sigma: V \to
TV$ satisfying
\begin{equation} \label{qdiff}
{\rm (Quantum\ Differential\ Equation)} \quad  \hbar \partial_\xi \sigma (v) =  \xi \star_{v} \sigma, \forall v,\xi \in
 V. \end{equation}
For a full discussion of the correspondence between Frobenius
manifolds and CohFT's the reader is referred to Manin \cite{man:fro}.
\end{enumerate}
\end{remark} 

\subsection{Complexified cyclohedron and 
 traces on CohFT algebras}
\label{traces}

In this section we study the moduli spaces of stable marked {\em
  parametrized} curves.  These are a special case of moduli spaces of
stable maps (the degree one case) but we prefer to view them in a
different way, as a special case of the Fulton-MacPherson construction
\cite{fm:compact}.  We then discuss the associated notion of {\em
  trace} on a CohFT algebra.  Let $C$ be a smooth connected projective
curve.

\begin{definition} 
\begin{enumerate}
\item {\rm (Parametrized nodal curves)} A {\em $C$-parametrized nodal
  curve} is a (possibly disconnected) nodal curve $\hat{C}$ equipped
  with a morphism $u: \hat{C} \to C$ of homology class $u_* [\hat{C}]
  = [C]$ and with the same arithmetic genus.  That is, $\hat{C}$ is
  the union of irreducible components $C_0,\ldots,C_r$ where $u$ maps
  the {\em principal component} $C_0$ isomorphically onto $C$ and $u$
  maps the other irreducible {\em bubble components} $C_1,\ldots, C_r$
  onto points.  Since the arithmetic genus of $\hat{C}$ is the same as
  that of $C$, the bubble components must be rational.  A {\em
    marking} of a $C$-parametrized curve is an $n$-tuple $\ul{z} =
  (z_1,\ldots,z_n)$ of points in $\hat{C}^n$ distinct from the nodes
  and each other. An {\em isomorphism} of such curves is an
  isomorphism of nodal curves which induces the identity on $C$.
\item {\rm (Stable parametrized curves)} A $C$-parametrized curve is
  {\em stable} if it has no infinitesimal automorphisms, that is, each
  non-principal irreducible component of $\hat{C}$ has at least three
  marked or nodal points.
\item {\rm (Rooted forests)} Any $C$-parametrized curve has a {\em
  combinatorial type} which is a forest $\Gamma$ (finite collection of
  trees) with a distinguished {\em root vertex} corresponding to the
  principal component and a labelling of the semiinfinite edges given
  by a bijection $l: \Edge_\infty(\Gamma) \to \{ 1, \ldots, n \}$.  A
  rooted forest is {\em stable} if it corresponds to a stable
  parametrized curve, that is, each non-root vertex has valence at
  least three.
\end{enumerate} 
\end{definition} \vskip .1in 

The set $\ovl{M}_n(C)$ of isomorphism classes of connected stable
$C$-parametrized curves has a natural topology, similar to that of
$\ovl{M}_{0,n}$ in genus zero: The following can be taken as a
definition or a proposition using a suitable construction of the
universal deformation of a stable map to $C$:

\begin{definition}  {\rm (Convergence of a sequence of parametrized stable curves)} 
Suppose $C$ has genus $0$.  A sequence $[( \hat{C}_\nu,u_\nu)]$ with
smooth domain $\hat{C}_\nu$ {\em converges} to $[(\hat{C},u)]$ if
there exists, for each irreducible component $\hat{C}_j$ of the limit
$\hat{C}$, a sequence of holomorphic isomorphisms $\phi_{j,\nu}:
\hat{C}_j \to \hat{C}_\nu$ such that
\begin{enumerate}
\item {\rm (Limit of a marking)}  for all $i,j$,
$\lim_{\nu \to \infty} \phi_{j,\nu}^{-1}(z_{i,\nu}) = z_i^j$, 
the node in $C_j$ connecting to the irreducible component of $C$
containing $z_i$, or $z_i$ if $z_i$ is contained in $C_j$;
\item {\rm (Limit of a different parametrization)}  for all $j \neq k$, 
$\lim_{\nu \to \infty} \phi_{j,\nu} \phi_{k,\nu}^{-1}$ has limit
  the constant map with value the node of $C_j$ connecting to $C_k$; and
\item {\rm (Limit of the map)} for all $j$, $\lim \phi_{j,\nu}^* u_\nu
  = u | C_j $.
\end{enumerate} 
Convergence for nodal domains $\hat{C}_\nu$ is defined similarly, by
considering convergence on each irreducible component separately.  
\end{definition} \vskip .1in 

The definition for arbitrary genus is similar, but the maps
$\phi_{j,\nu}$ exist only after removing small neighborhoods of the
nodes.  The topology on $\ovl{M}_n(C)$ induced by this notion of
convergence is compact and Hausdorff, and a special case of the
Fulton-MacPherson compactification of configuration spaces considered
in \cite{fm:compact}.  The open stratum $M_n(C)$ of $\ovl{M}_n(C)$ is
the configuration space $\Conf_n(C)$ of $n$-tuples of distinct points
on $C$.

More generally for any rooted forest $\Gamma$ with $n$ semiinfinite
edges we denote by $M_{n,\Gamma}(C)$ the moduli space of isomorphism
classes with type $\Gamma$ and by $\ovl{M}_{n,\Gamma}(C)$ its closure.
The moduli spaces of stable marked curves $\ovl{M}_{n,\Gamma}(C)$
satisfy a natural functoriality with respect to morphisms of rooted
forests $\Gamma$.  We say that a {\em morphism of rooted forests} is a
morphism of modular graphs corresponding to the rooted forests mapping
the root vertex to the root vertex.
\begin{proposition} {\rm (Morphisms of moduli spaces associated
to morphisms of forests)}
\begin{enumerate}
\item 
Any morphism of rooted forests $\Upsilon: \Gamma \to \Gamma'$ induces a
morphism of moduli spaces $\ovl{M}(\Upsilon): \ovl{M}_{n,\Gamma}(C) \to \ovl{M}_{n,\Gamma'}(C) .$ 
\item 
The boundary of $M_{n,\Gamma}(C)$ is the union of spaces
$M_{n,\Gamma'}(C)$ such that there is a morphism of rooted forests
$\Gamma' \to \Gamma$ collapsing an edge.
\item If $\Gamma'$ is obtained from $\Gamma$ by cutting an edge, then
  there is an isomorphism $\ovl{M}_{n,\Gamma'}(C) \to
  \ovl{M}_{n,\Gamma}(C)$ identifying the vertices corresponding to the
  additional semi-infinite edges.
\end{enumerate} 
\end{proposition}  
The proof is standard from properties of moduli spaces of stable maps.
Note that if $\Gamma' = \Gamma_0 \cup \Gamma_1$ is disconnected where
$\Gamma_0$ contains the root vertex then $\ovl{M}_{n,\Gamma}(C) \cong
\ovl{M}_{n_0,\Gamma_0}(C) \times \ovl{M}_{0,n_1,\Gamma_1}$ where $n_j$
is the number of semi-infinite edges of $\Gamma_j$.  Thus the boundary
of $\ovl{M}_{n,\Gamma}$ is the union of products of lower-dimensional
moduli spaces of $C$-parametrized stable curves and stable curves.

\begin{remark} {\rm (Relation to the cyclohedron)} 
Let $C$ be a projective line.  Any anti-holomorphic involution of $C$
induces an anti-holomorphic involution of $\ovl{M}_n(C)$, with fixed
point set $\ovl{M}_n(C)^\R$ identified with the moduli space of stable
parametrized $n$-marked disks.  The symmetric group $S_n$ acts by
permutation, and the quotient by $S_{n-1}$ is isomorphic to the subset
$\ovl{M}_n(C)^{\R,+}$ of $\ovl{M}_n(C)^\R$ such that the marked points
$z_0,\ldots,z_n$ occur in cyclic order around the boundary of the
disk.  The action of $S^1$ by rotation preserves $\ovl{M}_n(C)^{\R,+}$
and the quotient is the cyclohedron $\Cycl_n$, that is, the moduli
space of points on the circle compactified by allowing bubbling, see
Markl \cite{markl:free}.  In this sense it is slight abuse of
terminology to call $\ovl{M}_n(C)$ the complexification of $\Cycl_n$;
rather, $\ovl{M}_n(C)$ is the complexification of a circle bundle over
$\Cycl_n$.  
\end{remark} \vskip .1in

The boundary structure of the moduli space $\ovl{M}_n(C)$ is described
in the following.

\begin{proposition} The boundary of $\ovl{M}_n(C)$ is the union of the following subspaces
(which will be divisors once the algebraic structure on $\ovl{M}_n(C)$
  is introduced): For each subset $I \subset \{ 1, \ldots, n\}$ of
  order at least two a subspace $ \iota_I : D_I \to \ovl{M}_n(C) $
  where the markings for $i \in I$ have bubbled off onto an
  (unparametrized) sphere bubble.  The subspace $D_I$ admits a
  homeomorphism $ \varphi_I: D_I \to \ovl{M}_{0,|I|+1} \times \ovl{M}_{n
    - |I| + 1}(C) .$
\end{proposition} 

For any $\beta \in \ovl{M}_n(C)$, the pull-back $\iota_I^* \beta$ to a
subspace $D_I$ has a K\"unneth decomposition
\begin{equation} \label{kunneth2}
 \iota_I^* \beta = \sum_{j \in J} \beta_{1,j} \otimes
 \beta_{2,j}\end{equation}
for some index set $J$ and classes $\beta_{1,j} \in
H(\ovl{M}_{0,|I|+1})$ and $\beta_{2,j} \in \ovl{M}_{n-|I| + 1}(C)$.
In general, the moduli space of stable maps is not smooth.  However,
the space $\ovl{M}_n(C)$, as a special case of Fulton-MacPherson
\cite{fm:compact}, is a compact smooth manifold.  In particular, any
subset $D_I$ has a homology class $[D_I] \in H_2(\ovl{M}_n(C),\Z)$ and
a dual class $\gamma_I \in H^2(\ovl{M}_n(C),\Z)$, although we work with
rational coefficients below.  Let $\Lambda$ be a vector space.

\begin{definition}  {\rm (Trace on a CohFT algebra)} 
\label{cohfttrace} 
A ($C$-based, $\Lambda$-valued) {\em trace} on a CohFT algebra $V$ is
a collection of $S_n$-invariant (with Koszul signs) multilinear maps
$$ \tau^n: V^n \times H(\ovl{M}_n(C)) \to
\Lambda, \quad n \ge 0 $$
%
satisfying a splitting axiom for any $I \subset \{ 1,\ldots, n \}$
$$ \tau^n(\alpha; \beta \cup \gamma_I)
=  \tau^{n - |I| + 1}(\alpha_i, i \notin I, \mu^{|I|}(\alpha_i, i \in I; \cdot);
\cdot )(\iota_I^* \beta) $$
where $\gamma_I$ is the dual class to $D_I$ and the $\cdot$'s denote
insertion of the K\"unneth components of $\beta$.
\noindent That is, with $\beta$ as in \eqref{kunneth2},
$$ \tau^n(\alpha; \beta \cup \gamma_I ) = \sum_{j \in J}
 \tau^{n-|I| + 1}(\alpha_i, i \notin I, \mu^{|I|}(\alpha_i, i \in I; \beta_{1,j});
\beta_{2,j} ) .$$
\end{definition} \vskip .1in  
\begin{remark} 
\begin{enumerate}
\item 
In our main application, gauged Gromov-Witten invariants will define a
$\Lambda_X^G$-valued trace on $QH_G(X)$, which is a CohFT defined over
the field of fractions of $H(BG) \otimes \Lambda_X^G$.  In other
words, the space $\Lambda$ above need not be the ring or field over
which the CohFT algebra or CohFT is defined.
\item One should compare the notion of trace with that for \ainfty
  algebras described in \cite[Proposition 2.14]{markl:sac}.  The
  corresponding notion for an \ainfty space $X$ consists of a sequence
  of maps
$$ X^n \times \Cycl_n \to Y,\quad n \ge 0 $$
to an ordinary space $Y$, satisfying a suitable splitting axiom.
\end{enumerate} 
\end{remark} \vskip .1in

Any trace $(\tau^n)_{n \ge 0}$ on a CohFT algebra $V$ defines a formal
map
\begin{equation} \label{potential}
 \tau: V \to \Lambda, \quad v \mapsto \sum_{n} \frac{1}{n!}
 \tau^n(v,\ldots, v;1) \end{equation}
often called a {\em potential}.  The splitting axiom implies that the
second derivatives of $\tau$ with two point classes inserted define a
$\Lambda$-valued family of bilinear forms on $T_v V$ compatible with
the multiplications $\star_v$ on $T_v V$: 


\subsection{Complexified multiplihedron and morphisms of CohFT algebras}
\label{ziltener}

In this section we review a construction of Ma'u-Woodward
\cite{mau:mult}, based on earlier work of Ziltener \cite{zilt:phd},
which introduces a compactification of the moduli space of distinct
points on the affine line up to translation.  This compactification
``complexifies'' the multiplihedron in the same way that the
Grothendieck-Knudsen space and the Fulton-MacPherson spaces complexify
the associahedron and cyclohedron respectively.  We then discuss the
associated notion of {\em morphism} of CohFT algebras.  Let $\bA$
denote an affine line over $\C$, unique up to isomorphism.  We denote
by
$$\Omega^1(\bA,\C)^\C  = \{ w \d z | w \in \C  \} \cong \C$$ 
the space of $\C$-invariant one-forms on $\bA$.  

\begin{definition} {\rm (Scaled affine line)} \label{scale} A {\em
scaling} of an affine line $\bA$ is a translation-invariant, non-zero
  one form $\lambda \in \Omega^1(\bA,\C)^\C$.  A {\em scaled affine
    line} is an affine line equipped with a scaling.  An {\em
    $n$-marking} of an affine line is an $n$-tuple $\ul{z}=
  (z_1,\ldots, z_n)$ of distinct points in $\bA^n$.  An {\em
    isomorphism} of scaled $n$-marked affine lines is an affine
  isomorphism $\psi: C_0 \to C_1$, such that $\psi^* \lambda_1 =
  \lambda_0$ and $\psi(z_{0,i}) = z_{1,i}, i =1,\ldots, n$.
\end{definition} \vskip .1in 

\begin{figure}[ht]
\begin{center} 
\begin{picture}(0,0)%
\includegraphics{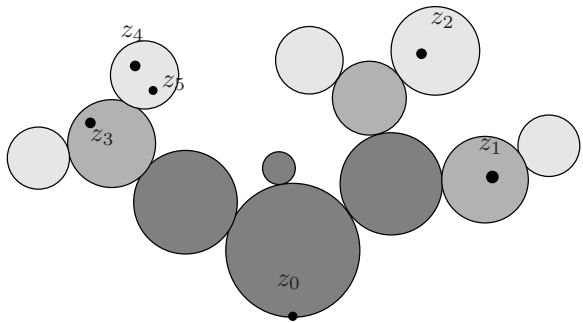}%
\end{picture}%
\setlength{\unitlength}{4144sp}%
\begingroup\makeatletter\ifx\SetFigFontNFSS\undefined%
\gdef\SetFigFontNFSS#1#2#3#4#5{%
  \reset@font\fontsize{#1}{#2pt}%
  \fontfamily{#3}\fontseries{#4}\fontshape{#5}%
  \selectfont}%
\fi\endgroup%
\begin{picture}(3441,1893)(2013,-3625)
\put(3635,-3434){\makebox(0,0)[lb]{{{{$z_0$}%
}}}}
\put(4840,-2624){\makebox(0,0)[lb]{{{{$z_1$}%
}}}}
\put(4552,-1868){\makebox(0,0)[lb]{{{{$z_2$}%
}}}}
\put(2518,-2571){\makebox(0,0)[lb]{{{{$z_3$}%
}}}}
\put(2699,-1948){\makebox(0,0)[lb]{{{{$z_4$}%
}}}}
\put(2946,-2243){\makebox(0,0)[lb]{{{{$z_5$}%
}}}}
\end{picture}%
\end{center}
\caption{An affine scaled curve}
\label{zilt}
\end{figure} 

Let $M_{n,1}(\bA)$ denote the moduli space of isomorphism classes of
scaled $n$-marked affine lines.  If $\bA$ is a scaled affine line then
the group of automorphisms of $\bA$ preserving the scaling is the
additive group $\C$ acting on $\bA$ by translation.  Thus

\begin{proposition}   The moduli space
$M_{n,1}(\bA)$ may be identified with the configuration space
  $\Conf_n(\bA)$ of $n$-tuples of distinct points on $\bA$ up to the
  action of $\C$ by translation, $ M_{n,1}(\bA) \cong \Conf_n(\bA)/\C
  .$
\end{proposition} 

\begin{proof}   For any tuple $(z_1,\ldots,z_n,\lambda)$ 
we take the unique rescaling so that $\lambda = \d z$ and then take
the associated configuration.  Conversely, any configuration defines
an affine scaled map by taking the scaling $\lambda = \d z$ to be
standard.
\end{proof} 

\begin{remark}  \label{twoform} {\rm (Two-forms instead of one-forms)} 
The moduli space $M_{n,1}(\bA)$ can be viewed in a different way: Any
scaling $\lambda$ gives rise to a real area form $\omega_{\bA} :=
\lambda \wedge \ovl{\lambda}$ on $\bA$.  Replacing $\lambda$ with
$\omega_{\bA}$ amounts to forgetting a complex phase; thus, one can
view ${M}_{n,1}(\bA)$ as the moduli space of data
$(z_1,\ldots,z_n,\omega,\phi)$ where $z_1,\ldots,z_n \in \bA$ are
distinct points, $\omega_{\bA} \in \Omega^2(\bA,\R)$ is a
translationally-invariant area form, and $\phi \in U(1)$ is a phase.
Any automorphism $\psi$ of $\bA$ has a well-defined {\em argument}
$\arg(\psi) \in U(1)$ giving the angle of rotation, and $\psi$ acts on
$(\omega_{\bA},\phi)$ by $(\psi^* \omega_{\bA}, \arg(\psi) \phi)$.
This is the point of view taken in Ziltener's thesis \cite{zilt:phd}.
\end{remark} \vskip .1in 

The moduli space $M_{n,1}(\bA)$ has a natural compactification
obtained by allowing bubbles with degenerate scalings.

\begin{definition} {\rm (Stable nodal scaled affine lines)}  
\label{affinescaled}
A {\em possibly degenerate scaling} on an affine line $\bA$ is an
element of the set
$$\ovl{\Omega}^1(\bA,\C)^\C = \Omega^1(\bA,\C)^\C \cup
\{ \infty \} \cong \P .$$
A possibly degenerate scaling $\lambda$ is {\em degenerate} if
$\lambda = 0 $ or $\lambda = \infty$ and is {\em non-degenerate}
otherwise.  The action of the group of automorphisms $\Aut(\bA)$ on
$\Omega^1(\bA,\C)^\C$ by pull-back extends naturally to an action on
$\ovl{\Omega}^1(\bA,\C)$, with fixed points $\{ 0 \}, \{ \infty \}$.  

Let $C$ be a nodal curve and $\omega_C$ the dualizing sheaf, whose
sections consist of one-forms possibly with poles at the nodes of $C$
whose residues on either side of a node are equal.  Let $\P(\omega_C
\oplus C)$ denote the projectivization of $\omega_C$.  Let $\lambda: C
\to \P(\omega_C \oplus \C)$ be a section.  A component $C_i$ of $C$
will be called {\em colored} if the the restriction of $\lambda$ to
$C_i$ is finite and non-zero.  A {\em nodal marked scaled affine line}
is a datum $(C,z_0,\ldots,z_n,\lambda)$ such that the following holds:
\begin{enumerate} 
\item[] (Monotonicity) on any non-self-crossing path from a marking
  $z_i$ to the root marking $z_0$, there is exactly one {\em colored
  irreducible component} with finite scaling; and the irreducible components before
  (resp. after) this irreducible component have infinite (resp. zero scaling).
\end{enumerate}
See Figure \ref{zilt}, where irreducible components with infinite
resp. finite, non-zero resp. zero scaling are shown with dark
resp. grey resp. light grey shading.  An {\em isomorphism} of nodal
marked scaled affine lines $(C_j,\ul{z}_j,\lambda_j),j = 0,1$ is an
isomorphism of nodal curves $\phi: C_0 \to C_1$ intertwining the
(possibly degenerate) scalings and markings in the sense that $\phi^*
\lambda_1 = \lambda_0$ and $\phi(z_{0,i}) = z_{1,i}$.  A nodal marked
scaled affine line is {\em stable} if it has no automorphisms, or
equivalently, if each irreducible component with finite scaling has at
least two special points, and each irreducible component with
degenerate scaling has at least three special points.
\end{definition} \vskip .1in


The space $\ovl{M}_{n,1}(\bA)$ of isomorphism classes of connected
stable scaled $n$-marked lines has a natural topology, similar to the
topology on the moduli space of stable curves.  Given a stable affine
scaled curve $(C,z_1,\ldots,z_n,\lambda)$, for any marking $z_i$ and
irreducible component $C_j$ we denote by $z_i^j$ the node in $C_j$
connecting to the irreducible component of $C$ containing $z_i$, or
$z_i$ if $z_i$ is contained in $C_j$.  The following can be taken as a
definition or a proposition with a suitable notion of family of stable
scaled marked lines, see \cite[Example 4.2]{qk2}.

\begin{definition}  {\rm (Convergence of a sequence of nodal scaled affine
lines)} A sequence $[(C_\nu,\ul{z}_\nu,\lambda_\nu)]$
  with smooth domain $C_\nu$ {\em converges} to
  $[(C,\ul{z},\lambda)]$ if there exists, for each irreducible component
  $C_j$ of the limit $C$, a sequence of holomorphic isomorphisms
  $\phi_{j,\nu}: C_j \to C_\nu$ such that
\begin{enumerate}
\item  {\rm (Limit of the scaling)}  $\lim_{\nu \to \infty} \phi_{j,\nu}^* \lambda_\nu =
\lambda |C_j$;
\item {\rm (Limit of a marking)} $\lim_{\nu \to \infty}
  \phi_{j,\nu}^{-1}(z_{i,\nu}) = z_i^j$; and 
\item {\rm (Limit of a different parametrization)}  $\lim_{\nu \to \infty} \phi_{j,\nu} \phi_{i,\nu}^{-1}$ has limit
  the constant map with value the node of $C_j$ connecting to $C_i$.
\end{enumerate} 
Convergence for sequences with nodal domain is defined similarly.
\end{definition} \vskip .1in 

\begin{example}  {\rm (Two markings converging)}   If $C_\nu = \P = \bA \cup \{ \infty \}$ 
and two points $z_{1,\nu}, z_{2,\nu}$ come together in the sense that
$\lim_{\nu \to \infty} z_{1,\nu} - z_{2,\nu} \to 0$, then there exists
a sequence of holomorphic maps $\phi_{\nu}: \P \to C_\nu$ such that
$\phi_{\nu}^{-1}(z_{1,\nu}), \phi_{\nu}^{-1}( z_{2,\nu} )$ converge to
distinct points, and the scaling $\phi_{\nu}^*(\lambda_\nu)$
converges to zero.  The limiting configuration consists of a
irreducible component with two markings and one node with zero
scaling, and an irreducible component with finite scaling, one node,
and the root marking $z_0$.  See Figure \ref{twoconverge}.
\end{example}

\begin{figure}[ht]
\begin{center} 
\begin{picture}(0,0)%
\includegraphics{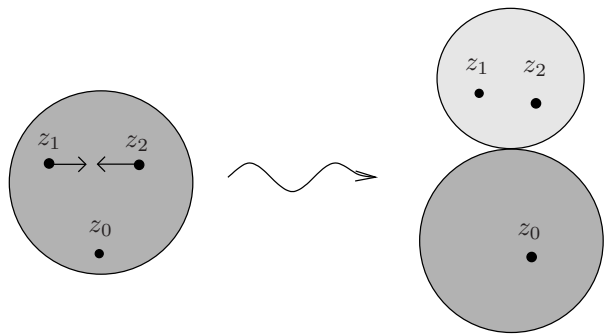}%
\end{picture}%
\setlength{\unitlength}{4144sp}%
\begingroup\makeatletter\ifx\SetFigFont\undefined%
\gdef\SetFigFont#1#2#3#4#5{%
  \reset@font\fontsize{#1}{#2pt}%
  \fontfamily{#3}\fontseries{#4}\fontshape{#5}%
  \selectfont}%
\fi\endgroup%
\begin{picture}(3569,1956)(3474,-4245)
\put(6522,-3677){\makebox(0,0)[lb]{{{{$z_0$}%
}}}}
\put(3955,-3660){\makebox(0,0)[lb]{{{{$z_0$}%
}}}}
\put(3647,-3137){\makebox(0,0)[lb]{{{{$z_1$}%
}}}}
\put(4172,-3148){\makebox(0,0)[lb]{{{{$z_2$}%
}}}}
\put(6207,-2687){\makebox(0,0)[lb]{{{{$z_1$}%
}}}}
\put(6553,-2699){\makebox(0,0)[lb]{{{{$z_2$}%
}}}}
\end{picture}%
\end{center} 
\caption{Two markings converging} 
\label{twoconverge}
\end{figure}

\begin{example} {\rm (Two markings diverging)} 
If $C_\nu = \P$ for all $\nu$ with constant scaling $\lambda_\nu$ and
two points $z_{1,\nu}, z_{2,\nu}$ go to infinity in $\bA \subset \P$
in different directions, then for $k \in \{1, 2 \}$ there exists (i) a
sequence of holomorphic maps $\phi_{k,\nu}: \C \to C_\nu$ such that
$\phi_{k,\nu}^{-1}(z_{k,\nu})$ and $\phi_{k,\nu}^* \lambda_\nu$
converge for $k \in \{ 1, 2\}$ and (ii) a sequence $\phi_{12,\nu}: \C
\to C_\nu$ such that $\phi_{12,\nu}^{-1} z_{k,\nu}$ for $k = 1,2$
converge to distinct points and $\phi_{12,\nu}^* \lambda_\nu$
converges to infinity.  The limiting configuration consists of two
components with a single marking and node and finite scaling, and a
component with two nodes, the root marking, and infinite scaling.  See
Figure \ref{twodiverge}.
\end{example} 

\begin{figure}[ht]
\begin{center} 
\begin{picture}(0,0)%
\includegraphics{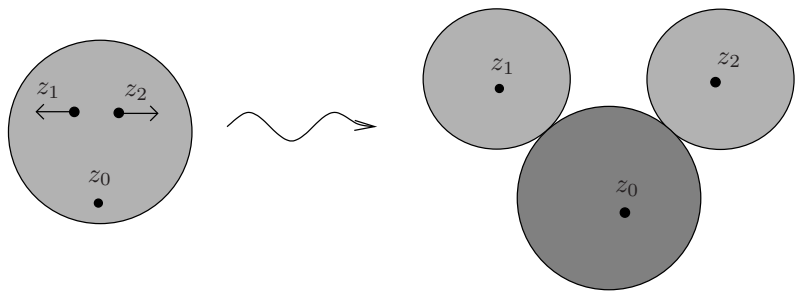}%
\end{picture}%
\setlength{\unitlength}{4144sp}%
\begingroup\makeatletter\ifx\SetFigFont\undefined%
\gdef\SetFigFont#1#2#3#4#5{%
  \reset@font\fontsize{#1}{#2pt}%
  \fontfamily{#3}\fontseries{#4}\fontshape{#5}%
  \selectfont}%
\fi\endgroup%
\begin{picture}(4716,1697)(3474,-4292)
\put(3955,-3660){\makebox(0,0)[lb]{{{{$z_0$}%
}}}}
\put(3647,-3137){\makebox(0,0)[lb]{{{{$z_1$}%
}}}}
\put(4172,-3148){\makebox(0,0)[lb]{{{{$z_2$}%
}}}}
\put(7111,-3724){\makebox(0,0)[lb]{{{{$z_0$}%
}}}}
\put(6366,-2992){\makebox(0,0)[lb]{{{{$z_1$}%
}}}}
\put(7719,-2954){\makebox(0,0)[lb]{{{{$z_2$}%
}}}}
\end{picture}%
\end{center}
\caption{Two markings diverging} 
\label{twodiverge}
\end{figure}

We denote by $\ovl{M}_{n,1}(\bA)$ the space of isomorphism classes of
connected nodal scaled lines, equipped with the topology above.  By
Ma'u-Woodward \cite{mau:mult} $\ovl{M}_{n,1}(\bA)$ is a compact
Hausdorff space.

\begin{definition}  {\rm (Combinatorial types of nodal scaled affine lines)} 
The {\em combinatorial type} of a connected scaled affine line is a
{\em colored tree} consisting of a tree $\Gamma =
(\Ve(\Gamma),\Edge(\Gamma))$ together with a partition of the vertices
$$\Ve(\Gamma) = \Ve^0(\Gamma) \cup \Ve^1(\Gamma) \cup
\Ve^\infty(\Gamma)$$
and a labelling of its semi-infinite edges given by a bijection
$\Edge_\infty(\Gamma) \to \{ 1, \ldots, n \}$ that satisfies the
combinatorial version of the monotonicity condition:
\begin{enumerate} 
\item[] (Monotonicity) on any non-self-crossing path from a
  semi-infinite edge labelled $j$ to the semi-infinite edge labelled
  $0$, there is exactly one vertex in $\Ve^1(\Gamma)$, all vertices
  before resp. after are in $\Ve^0(\Gamma)$
  resp. $\Ve^\infty(\Gamma)$.
\end{enumerate}
A colored tree is {\em stable} if it corresponds to a stable affine
scaled line, that is, each vertex $v \in \Ve^0(\Gamma)$
resp. $\Ve^\infty(\Gamma)$ resp. $\Ve^1(\Gamma)$ has valence at least
$3$ resp. $3$ resp. $2$.
\end{definition} \vskip .1in 

We call the vertices in $\Ve^1(\Gamma)$ the {\em colored vertices}.
Colored trees can be pictured as trees where part of the tree
containing the root semi-infinite edge edge is ``below water'' and
part ``above water''; the colored vertices in $\Ve^1(\Gamma)$ are
those ``at the water level''.  The monotonicity condition then says
that the path from any ``above water'' semiinfinite edge to the unique
``below water'' semiinfinite edge passes through the water surface
exactly once.  However, in our trees we adopt the standard convention
of having the root edge (which corresponds to an outgoing marking) at
the top of the picture.

More generally we allow disconnected curves where each connected
component is either a nodal affine curve, or a nodal curve with
infinite or zero scaling.

\begin{definition}  {\rm (Colored forests)}
A colored forests $\Gamma$ is a union of components that are either
colored trees, or ordinary trees with all vertices in $\Ve^0(\Gamma)$
or all vertices in $\Ve^\infty(\Gamma)$.
\end{definition} \vskip .1in 


For any colored forest $\Gamma$ with $n$ semiinfinite edges we denote
by ${M}_{n,1,\Gamma}(\bA)$ the space of isomorphism classes of scaled
lines of combinatorial type $\Gamma$, and $\ovl{M}_{n,1,\Gamma}(\bA)$
its closure.

\begin{definition} {\rm (Morphisms of colored forests)}  
A {\em morphism of colored forests} from $\Gamma$ to $\Gamma'$ is a
combination of the following simple morphisms:
\begin{enumerate} 
\item {\rm (Collapsing edges without relations)} $\Upsilon: \Gamma \to
  \Gamma'$ {\em collapses an edge} if $\Upsilon$ is injective except
  that it maps two vertices in $\Ve^0(\Gamma) \cup \Ve^1(\Gamma)$ to
  the same vertex in $\Ve(\Gamma')$ or two vertices in
  $\Ve^\infty(\Gamma)$ with the same vertex in $\Ve^\infty(\Gamma)$.
  (In other words, any edge except those connecting $\Ve^1(\Gamma)$
  with $\Ve^\infty(\Gamma))$.
\begin{figure}[ht]
\begin{center} 
\begin{picture}(0,0)%
\includegraphics{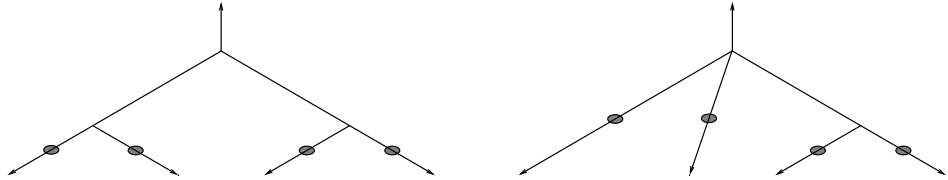}%
\end{picture}%
\setlength{\unitlength}{3947sp}%
\begingroup\makeatletter\ifx\SetFigFontNFSS\undefined%
\gdef\SetFigFontNFSS#1#2#3#4#5{%
  \reset@font\fontsize{#1}{#2pt}%
  \fontfamily{#3}\fontseries{#4}\fontshape{#5}%
  \selectfont}%
\fi\endgroup%
\begin{picture}(5921,1117)(2389,-866)
\end{picture}%
\end{center} 
\caption{Collapsing an edge connecting two vertices of the same type}
\label{coloredfig3}
\end{figure}

\item {\rm (Collapsing edges with relations)} $\Upsilon: \Gamma \to
  \Gamma'$ {\em collapses edges} if $\Upsilon$ is injective except for
  one colored vertex of $\Gamma'$ whose inverse image in $\Ve(\Gamma)$
  is a collection of colored vertices in $\Gamma$ and a single vertex
  in $\Ve^\infty(\Gamma)$, joined to each of the colored vertices by a
  single edge.  Note that one cannot write such a morphism as a
  composition of morphisms each collapsing a single edge, since there
  is no way to assign the coloring of vertices of the resulting graph
  which results in a colored forest.
\begin{figure}[ht]
\begin{center} 
\begin{picture}(0,0)%
\includegraphics{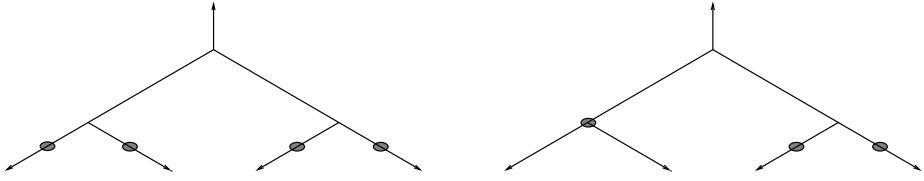}%
\end{picture}%
\setlength{\unitlength}{3947sp}%
\begingroup\makeatletter\ifx\SetFigFontNFSS\undefined%
\gdef\SetFigFontNFSS#1#2#3#4#5{%
  \reset@font\fontsize{#1}{#2pt}%
  \fontfamily{#3}\fontseries{#4}\fontshape{#5}%
  \selectfont}%
\fi\endgroup%
\begin{picture}(5786,1092)(2389,-841)
\end{picture}%
\end{center} 
\caption{Collapsing edges with relations}
\label{coloredfig4}
\end{figure}

\item {\rm (Cutting an edge without relations)} $\Upsilon: \Gamma \to
  \Gamma'$ {\em cuts an edge} of $\Gamma$ if the vertices are the
  same, but $\Gamma'$ has one fewer edge than $\Gamma$ and the edge
  does not lie between the colored vertices and the root edge.

\begin{figure}[ht]
\includegraphics[width=5in]{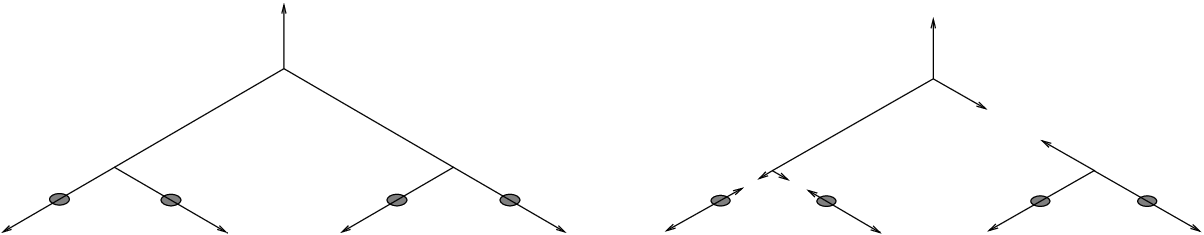}
\caption{Cutting edges with relations}
\label{cutwrelns}
\end{figure}

\item {\rm (Cutting edges with relations)} $\Upsilon: \Gamma \to
  \Gamma'$ {\em cuts edges with relations} of $\Gamma$ if the vertices
  are the same, but $\Gamma'$ has fewer edges than $\Gamma$, with each
  removed edge lying on a path between the root edge and the colored
  vertices, and each path passing through a unique such edge.  See
  Figure \ref{cutwrelns}.

\item {\rm (Forgetting tails)} $\Upsilon: \Gamma \to \Gamma'$ {\em
  forgets a tail} (semiinfinite edge) and any vertices that 
become unstable, recursively starting from the semiinfinite edges
furthest away from the root edge.    

\begin{figure}[ht]
\begin{center} 
\begin{picture}(0,0)%
\includegraphics{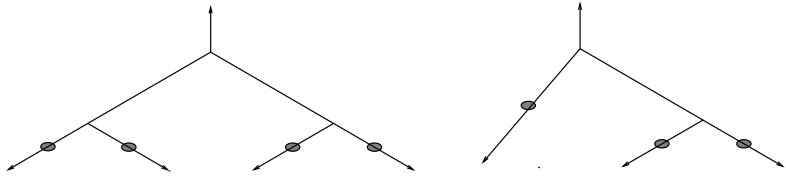}%
\end{picture}%
\setlength{\unitlength}{3947sp}%
\begingroup\makeatletter\ifx\SetFigFontNFSS\undefined%
\gdef\SetFigFontNFSS#1#2#3#4#5{%
  \reset@font\fontsize{#1}{#2pt}%
  \fontfamily{#3}\fontseries{#4}\fontshape{#5}%
  \selectfont}%
\fi\endgroup%
\begin{picture}(4916,1092)(2389,-803)
\end{picture}%
\end{center} 
\caption{Forgetting a tail and collapsing}
\label{coloredfig5}
\end{figure}

\end{enumerate} 
\end{definition} \vskip .1in 

\begin{remark} {\rm (More explanation on forgetting tails)} 
Forgetting a tail leaves possibly only the vertex adjacent to the tail
unstable, if it is colored with a single other edge adjacent, or
non-colored with two other adjacent edges.  In the first case,
removing that vertex and the other adjacent edge still leaves a
non-colored vertex which may be unstable, since it has one fewer edge.
If unstable, removing this vertex and identifying the other two edges
gives a stable colored tree.  In the second case, removing the vertex
gives a stable colored tree.  
\end{remark} 

%
%
%

\begin{proposition}  {\rm (Morphisms of moduli spaces induced by morphisms
of colored forests)} 
To any morphism $\Upsilon$ of colored forests $\Gamma \to \Gamma'$ one
can associate a morphism $\ovl{M}_n(\Upsilon): \ovl{M}_{n,1,\Gamma}(\bA)
\to \ovl{M}_{n,1,\Gamma'}(\bA) $ as follows.
\begin{enumerate} 
\item {\rm (Collapsing edges without relations)} Any morphism
  $\Upsilon: \Gamma \to \Gamma'$ collapsing an edge induces an
  inclusion $\ovl{M}(\Upsilon): \ovl{M}_{n,1,\Gamma}(\bA) \to
  \ovl{M}_{n,1,\Gamma'}(\bA)$.
\item {\rm (Collapsing edges with relations)} Any morphism $\Upsilon:
  \Gamma \to \Gamma'$ collapsing edges with relations induces an
  inclusion $\ovl{M}(\Upsilon): \ovl{M}_{n,1,\Gamma}(\bA) \to
  \ovl{M}_{n,1,\Gamma'}(\bA)$.
\item {\rm (Cutting an edge or edges with relations)} Any morphism
  $\Upsilon: \Gamma \to \Gamma'$ cutting an edge or edges with
  relations of $\Gamma$ induces a homeomorphism from
  $\ovl{M}_{n,1,\Gamma}(\bA)$ to $\ovl{M}_{n,1,\Gamma'}(\bA)$ by
  identifying the markings corresponding to the additional
  semiinfinite edges.
\item {\rm (Forgetting tails)} Any morphism $\Upsilon: \Gamma \to
  \Gamma'$ {\em forgetting a tail} induces a map $\ovl{M}(\Upsilon):
  \ovl{M}_{n,1,\Gamma}(\bA) \to \ovl{M}_{n,1,\Gamma'}(\bA)$ which
  forgets the corresponding marking and collapses any unstable
  components recursively starting with the semiinfinite edges
  corresponding to the finite markings.
\end{enumerate} 
\end{proposition} 

\begin{proof}   The existence of these maps is immediate from the definitions except
for the existence of the morphism for forgetting tails, which requires
an inductive argument collapsing the unstable components.  For this
note that forgetting a tail leaves possibly only the component
containing the corresponding marking unstable, if it had either
non-degenerate scaling and two special points, or degenerate scaling
and three special points.  Removing that component, and if the second
possibility holds, replacing the node with the remaining marking or
identifying the two remaining nodes, produces a new curve with one
fewer irreducible component.  In the second case, the resulting curve
is automatically stable.  In the first case, the adjacent irreducible
component has one fewer special point, and so now may be unstable.  If
so, removing that component, and either (i) identifying the nodes, if
the two special points were nodes, or (ii) changing the node to a
marking, if the two special points were a node and a marking, produces
a stable scaled affine curve.  We give a stronger, algebraic version
of the forgetful morphism in \cite[Example 4.2]{qk2}.
\end{proof}

\begin{lemma}  
For each $\Gamma$, the boundary of $\ovl{M}_{n,1,\Gamma}(\bA)$ consists
of those moduli spaces $M_{n,1,\Gamma'}(\bA)$ such that $\Gamma$ is
obtained from $\Gamma'$ by collapsing an edge or edges with relations.
Furthermore, each $\ovl{M}_{n,1,\Gamma}(\bA)$ is a product of the
moduli spaces $\ovl{M}_{n_j,1}(\bA)$ and $\ovl{M}_{0,n_j}$ corresponding
to the vertices of $\Gamma$, where $n_j$ are the valences.
\end{lemma} 

Ma'u-Woodward \cite{mau:mult} shows that the compactification
$\ovl{M}_{n,1}(\bA)$ has the structure of a projective variety, locally
isomorphic to a toric variety.  The local structure of
$\ovl{M}_{n,1}(\bA)$ near the stratum $M_{n,1,\Gamma}(\bA)$ of nodal
lines with combinatorial type $\Gamma$ may be described as follows.

\begin{definition} {\rm (Balanced labellings)} 
For any colored tree $\Gamma$, a labelling 
$$ \gamma: \Edge_{<
  \infty}(\Gamma) \to \C $$ 
is {\em balanced} iff
\begin{equation} \label{balanced}
\prod_{e \in P_{vw}} \gamma(e)^\pm = 1 \  \end{equation} 
where $v,w$ range over elements of $\Ve^1(\Gamma)$ and $P_{vw}$ is the
unique non-self-crossing path from $v$ to $w$, and in the product the
sign is positive if $e$ is pointing towards the root edge marked $z_0$
and negative otherwise.
\end{definition} \vskip .1in 

Let $Z_\Gamma \subset \Map(\Edge_{< \infty}(\Gamma),\C)$ denote the
space of balanced labellings.  An element of $Z_\Gamma$ is called a
tuple of {\em gluing parameters}.

\begin{example} {\rm (A singularity in the moduli space)} 
For the tree $\Gamma$ in Figure \ref{cartierex2}, with large dots
indicating vertices in $\Ve^1(\Gamma)$, the relations are $\gamma_3 =
\gamma_4, \gamma_1 \gamma_3 = \gamma_2 \gamma_5, \gamma_5 = \gamma_6
$.  The corresponding toric variety $Z_\Gamma$ corresponds to a
$3$-dimensional cone with $4$ extremal rays, and so the moduli space
has a singularity at the vertex.   \end{example}

\begin{figure}[ht]
\begin{center} 
\begin{picture}(0,0)%
\includegraphics{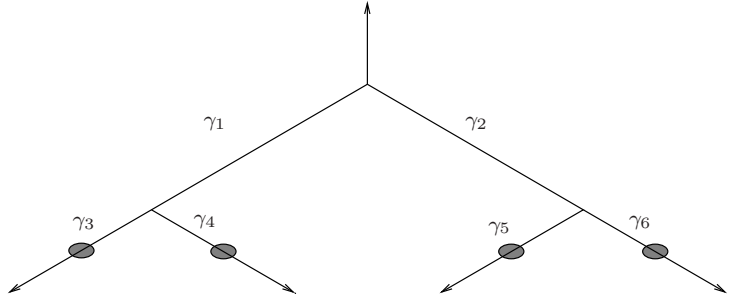}%
\end{picture}%
\setlength{\unitlength}{3947sp}%
\begingroup\makeatletter\ifx\SetFigFont\undefined%
\gdef\SetFigFont#1#2#3#4#5{%
  \reset@font\fontsize{#1}{#2pt}%
  \fontfamily{#3}\fontseries{#4}\fontshape{#5}%
  \selectfont}%
\fi\endgroup%
\begin{picture}(4544,1864)(2389,-1613)
\put(5277,-550){\makebox(0,0)[lb]{\smash{{\SetFigFont{8}{9.6}{\rmdefault}{\mddefault}{\updefault}{$\gamma_2$}%
}}}}
\put(2812,-1180){\makebox(0,0)[lb]{\smash{{\SetFigFont{8}{9.6}{\rmdefault}{\mddefault}{\updefault}{$\gamma_3$}%
}}}}
\put(3572,-1167){\makebox(0,0)[lb]{\smash{{\SetFigFont{8}{9.6}{\rmdefault}{\mddefault}{\updefault}{$\gamma_4$}%
}}}}
\put(5421,-1207){\makebox(0,0)[lb]{\smash{{\SetFigFont{8}{9.6}{\rmdefault}{\mddefault}{\updefault}{$\gamma_5$}%
}}}}
\put(6305,-1172){\makebox(0,0)[lb]{\smash{{\SetFigFont{8}{9.6}{\rmdefault}{\mddefault}{\updefault}{$\gamma_6$}%
}}}}
\put(3634,-550){\makebox(0,0)[lb]{\smash{{\SetFigFont{8}{9.6}{\rmdefault}{\mddefault}{\updefault}{$\gamma_1$}%
}}}}
\end{picture}%
\end{center} 
\caption{An example of a colored tree}
\label{cartierex2}
\end{figure}

\begin{proposition} \cite{mau:mult} \label{MAloc2}  There exists an open neighborhood
of $M_{n,1,\Gamma}(\bA)$ in $\ovl{M}_{n,1}(\bA)$ isomorphic to an
open neighborhood of $0 \times M_{n,1,\Gamma}(\bA)$ in $Z_\Gamma
\times M_{n,1,\Gamma}(\bA) .$
\end{proposition}  

\begin{proof}   The construction is a version of the 
construction of the universal deformation of a genus zero nodal curve
\cite[p. 184]{ar:alg2} in which small balls around the nodes are
removed and the components glued together via maps $z \mapsto
\gamma/z$.  The scaling is determined by the product of the gluing
parameters from the root component to the irreducible components with
finite scaling, independent of the choice of irreducible component
with finite scaling by the balanced condition \eqref{balanced}.
\end{proof} 

\begin{remark} \label{dimcount} {\rm (Codimension formula)} The codimension of a stratum $\ovl{M}_{n,1}(\bA)$ corresponding to a colored tree
$\Gamma$ is {\em not} the number of finite edges, but rather
$$ \codim( M_{n,1,\Gamma}(\bA)) = \# \Edge_{< \infty}(\Gamma) +1 -  
\# \Ver^1(\Gamma) $$
where the extra summand $1 - \Ver^1(\Gamma)$ corresponds to the minus
the number of relations on the gluing parameters $(\gamma(e))_{e \in \Edge_{< \infty}(\Gamma)} \in
Z_\Gamma$.  
\end{remark} 

\begin{proposition}  The boundary of $\ovl{M}_{n,1}(\bA)$ consists
of the following subsets (which will be {\em divisors} with respect to
the algebraic structure on $\ovl{M}_{n,1}(\bA)$ introduced later):
\begin{enumerate} 
\item {\rm (Bubbling points)} For any $I \subset \{1,\ldots, n\}$ of order
  at least two the subset
$$\iota_I: D_I \to \ovl{M}_{n,1}(\bA)$$ 
corresponding to the formation of a single bubble containing the
markings $I$, with an isomorphism
\begin{equation} \label{homeo1} 
 D_I \to \ovl{M}_{0,|I| + 1} \times \ovl{M}_{n - |I|+1,1}(\bA)
 .\end{equation}
\item {\rm (Blowing up scaling)} For any $r > 0$ and unordered
  partition $[I_1,\ldots, I_r]$, $I_1 \cup \ldots \cup I_r = \{
  1,\ldots n \}$ of order at least two, with each $I_j$ non-empty, a
  subset $D_{[I_1,\ldots,I_r]}$ corresponding to the formation of $r$
  bubbles with markings $I_1,\ldots, I_r$, attached to a remaining
  component with infinite scaling.  The map forgetting all but the
  infinitely-scaled locus induces an isomorphism
\footnote{Equation
    \eqref{homeo2} is corrected from the published version.}
  \begin{equation} \label{homeo2} D_{[I_1,\ldots,I_r] } \to
    \ovl{M}_{r,1}(\bA) \times \left( \prod_{i=1}^r
      \ovl{M}_{|I_i|,1}(\bA) \right).
 \end{equation}
\end{enumerate}
\end{proposition} 

\begin{remark}
\begin{enumerate} 
\item The inclusions of these subspaces give the collection of spaces
$\ovl{M}_{n,1}(\bA)$ the structure of an algebra over the operad
associated to the notion of {\em homotopy morphism} of operads in
\cite{ma:op}.  However, we will not use or need this language
and will not discuss it further. 
\item {\rm (Relation to the multiplihedon)} 
The moduli space $\ovl{M}_{n,1}(\bA)$ has a ``positive real
  locus'' that appears in Stasheff's description of \ainfty morphisms
  \cite{st:hs}.  Namely, taking a real structure on $\bA$, the
  anti-holomorphic involution on $\bA$ induces an anti-holomorphic
  involution of $\ovl{M}_{n,1}(\bA)$. We denote by
  $\ovl{M}_{n,1}(\bA)^\R$ the fixed point locus, in which all markings
  are on the real line.  The symmetric group $S_n$ acts on
  $\ovl{M}_{n,1}(\bA)$, and restricts to an action on
  $\ovl{M}_{n,1}(\bA)^\R$ with fundamental domain given as the closure
  of the subset $M_{n,1}(\bA)^{\R,+}$ where $z_1 < z_2 < \ldots <z_n$,
  homeomorphic to Stasheff's multiplihedron $\Mult_n$ \cite{mau:mult}.
An \ainfty morphism of \ainfty spaces $X,Y$ consists of a sequence 
of maps 
$$ X^n \times \Mult_n \to Y, n \ge 0 $$
satisfying a suitable splitting axiom on the boundary.
\end{enumerate}
\end{remark} \vskip .1in  

The splitting axiom for morphisms of CohFT algebras is defined via
divisors on $\ovl{M}_{n,1}(\bA)$.  Because the singularities of the
toric variety $Z_\Gamma$ occur in complex codimension at least three,
$\ovl{M}_{n,1}(\bA)$ has a unique homology class of top dimension.  In
particular, each of the boundary divisors above has a well-defined
homology class in $\ovl{M}_{n,1}(\bA)$.  However, $\ovl{M}_{n,1}(\bA)$
is not smooth (and not a rational homology manifold) and not every
boundary stratum has a dual class.  That is, given a divisor
\begin{equation} \label{coeff}
 D = \sum_I n_I D_I + \sum_{r,[I_1,\ldots,I_r]} n_{[I_1,\ldots,I_r]} D_{[I_1,\ldots,I_r] }
\end{equation}
there may or may not exist a class $\gamma \in H^2(\ovl{M}_{n,1}(\bA))$
that satisfies
$$ \langle \beta, [D] \rangle = \langle \beta \cup \gamma ,
[\ovl{M}_{n,1}(\bA)] \rangle .$$
This requires the restriction to combinations of boundary divisors
that have dual classes in the following definition. 

Let $(V,(\mu_V^n)_{n \ge 2})$ and $(W,(\mu_W^n)_{n \ge 2})$ be CohFT
algebras.

\begin{definition}  \label{morphismcohfts} A {\em morphism of
CohFT algebras} from $V$ to $W$ is a collection of $S_n$-invariant
  (with Koszul signs) multilinear maps
$$ \phi^n: V^n \times H(\ovl{M}_{n,1}(\bA)) \to W, \quad n
  \ge 0 $$
such that for any divisor $D$ of the form \eqref{coeff} with dual
class $\gamma \in H^2(\ovl{M}_{n,1}(\bA))$ we have
\begin{multline} \label{weaksplit}
\phi^n(\alpha, \beta \cup \gamma) = \sum_{I} n_I \phi^{n - |I|+1} (
\mu_V^{|I|}(\alpha_i, i \in I; \cdot ) ,\alpha_j, j \notin I;
\cdot)(\iota_I^* \beta) \\ + \sum_{r, [I_1,\ldots,I_r]} n_{[I_1,
  \ldots,I_r]} \mu_W^r( \phi^{I_1}( \alpha_i, i \in I_1; \cdot),
\ldots, \phi^{I_r}( \alpha_i, i \in I_r; \cdot); \cdot ) 
(\iota_{[I_1,\ldots,I_r]}^* \beta),
\end{multline}
where the sum is over unordered partitions $[I_1, \ldots , I_r]$ of
$\{1 ,\ldots ,n \}$ with some $I_j$ possibly empty, $\cdot$ indicates
insertion of the K\"unneth components of $\iota_I^* \beta$,
$\iota_{[I_1,\ldots,I_r]}^* \beta$, using the homeomorphisms
\eqref{homeo1}, \eqref{homeo2}, the sum on the right-hand-side is
assumed finite, and by convention if $n = 0$ we replace
$H(\ovl{M}_{n,1}(\bA))$ with $\Lambda$ (since in this case
$\ovl{M}_{n,1}(\bA)$ is empty).  A morphism of CohFT algebras $\phi$ is
{\em flat} resp. {\em curved} if $\phi^0$ is zero resp. non-zero.
\end{definition} \vskip .1in
\noindent 

\begin{remark}  \label{Drem}
\begin{enumerate} 
\item In our examples, $\Lambda = \cup_{a \in \R} \Lambda_{\ge a}$
  will be a filtered ring, and $V = \cup_{a \in \R} V_{\ge a}$, $W =
  \cup_{a \in \R} W_{\ge a}$ filtered $\Lambda$-modules.  We say that
  a {\em morphism of filtered CohFT algebras} is defined as above but
  where the right-hand-side of \eqref{weaksplit} is finite modulo
  $W_{\ge a}$ for any $a \in \R$.
\item See Nguyen-Woodward-Ziltener \cite{cartier} for a description of
  the space of Cartier divisors in $\ovl{M}_{n,1}(\bA)$, that is, a
  description of which combinations of codimension two strata have
  dual classes.
\item 
The definition of flat morphism of CohFT algebras (which has nothing
to do with flat morphism of rings etc.)  is analogous to the
definition of flat morphism of \ainfty algebras in \cite{fooo}.  That
is, $\phi^0$ is analogous to the {\em curvature} of an $A_\infty$
morphism.
\item 
The divisors $D_{\{ 1 \}, \{ 2 \}}, D_{ \{ 1, 2 \}} \subset
\ovl{M}_{2,1} \cong \P$ are points (the limiting points in Figures
\ref{twoconverge}, \ref{twodiverge}) and so have the same homology
class.  Using the splitting axiom \eqref{weaksplit} this implies that
if $\phi^0$ vanishes then $\phi^1$ is a homomorphism from
$(V,\mu^2_V)$ to $(W,\mu^2_W)$.  This is an analog of the fact that a
flat \ainfty morphism induces an algebra homomorphism of cohomology
groups.  
\item For simplicity we will consider here only the even case, that
  is, $V$ is a usual vector space and there are no signs. 
\end{enumerate}
\end{remark} \vskip .1in 

Now we discuss the connection of morphisms of CohFT algebras with
Frobenius manifolds, or rather, the underlying family of algebras:

\begin{definition}   Let $V, W$ be vector spaces equipped
with associative products $\star_v: T_vV^2 \to T_v V, \star_w: T_wW^2
\to T_w W$ varying smoothly in $v,w$.  A {\em $\star$-morphism} from
$V$ to $W$ is an analytic map $\phi: V \to W$ whose derivative $D_v
\phi$ is a morphism of algebras from $T_v V $ to $T_{\phi(v)}W $ for
all $v \in V$.
\end{definition} \vskip .1in 

In particular, if $\phi: V \to W$ is a $\star$-morphism with Taylor
coefficients $\phi^n$ then $D_0 \phi = \phi^1$ is an algebra
homomorphism from $T_0 V$ to $T_{\phi^0(1)} W$.

\begin{proposition}
Any morphism of CohFT algebras $( \phi^n )_{n \ge 0}$ from $V$ to $W$
defines a formal $\star$-morphism from $V$ to $W$ via the formula
$$ \phi: V \to W, \quad v \mapsto \sum_{n \ge
  0} \frac{1}{n!} \phi^n(v,\ldots, v) .$$ 
$\phi$ arises from a flat morphism of CohFT algebras iff $\phi(0) =0$.
\end{proposition} 

\begin{proof}  For convenience, we reproduce the argument from \cite[Proposition 2.43]{cartier}. 
Consider the relation
$ [D_{ \{ 1,2 \} }] = [D_{ \{ 1 \}, \{ 2 \}}] \in
H^2(\ovl{M}_{2,1}(\bA)) $
from Remark \ref{Drem} (d).  Its pull-back under the morphism
$\ovl{M}_{n,1}(\bA) \to \ovl{M}_{2,1}(\bA)$ forgetting all but the first
two markings is the relation
\begin{equation} \label{pullback} 
\sum_{r, [I_1,\ldots,I_r]} [D_{[I_1,I_2,\ldots, I_r]}] = \sum_{I}
    [D_I]
\end{equation} 
where the first sum is over partitions $I_1,\ldots, I_r$ with $1 \in
I_1,2 \in I_2$, and the second is over subsets $I \subset \{1,\ldots,
n \}$ with $\{ 1,2 \} \subset I$.  Indeed, the pull-back of the
coordinate on $\ovl{M}_{2,1}(\bA)$ to $\ovl{M}_{n,1}(\bA)$ is equal to
any of the (equal) gluing parameters at the nodes at a generic point
in any divisor $D_{[I_1,I_2,\ldots, I_r]}$ appearing on the
left-hand-side, see Ma'u-Woodward \cite{mau:mult}, and so has a zero
of order one at that divisor.  On the other hand, the pull-back is the
inverse of the gluing parameter at a generic point in any divisor
$D_I$ appearing on the right-hand side, and so has a pole of order one
at that divisor.  The splitting axiom implies for each $a,b \in T_v V$
\begin{eqnarray*}
D_v \phi(a \star_v b) 
&=&\sum_{n,i}\frac{1}{(i-2)!(n-i)!}\phi^{n-i+1}\big(\mu^i_V(a,b,v,\ldots,v;1),v,\ldots,v;1\big)\\
&=&\sum_{n,I} ((n-2)!)^{-1} \phi^{n-|I|+1}\big(\mu^{|I|}_V(a,b,v,\ldots,v;1),v,\ldots,v;1\big)\\
&=& \sum_{I_1 \ni 1, I_2 \ni 2, I_3,\ldots,I_r} 
((n-2)! \# \{j \, | \, I_j = \emptyset \} !)^{-1} 
\mu^r_W \big(
\phi^{|I_1|}(a,v,\ldots,v;1),\\
&&\phi^{|I_2|}(b,v,\ldots,v;1),\phi^{|I_3|}(v,\ldots,v;1),\ldots,\phi^{|I_r|}(v,\ldots,v;1);1\big) \\
&=& \sum_{i_1,i_2 \ge 1,i_3 \ldots,i_r\geq 0}\frac{1}{(i_1-1)!(i_2-1)!i_3!\cdots i_r!(r-2)!}\mu_W^r\big(\phi^{i_1}(a,v,\ldots,v;1),\\
&&\phi^{i_2}(b,v,\ldots,v;1),\phi^{i_3}(v,\ldots,v;1),\ldots,\phi^{i_r}(v,\ldots,v;1);1\big) \\
&=&\sum_{r}
\frac{1}{(r-2)!}\mu_W^r\big(D_v \phi(a),D_v \phi(b),\phi(v),\ldots, \phi(v);1) \\ 
&=& D_v \phi(a) \star_{\phi(v)} D_v \phi(b) 
\end{eqnarray*}
where the right-hand-side is defined formally, that is, via Taylor
series.  By definition $\phi(0) = 0$ iff $\phi^0$ vanishes iff
$(\phi^n)_{n \ge 0}$ is flat.
\end{proof} 

\subsection{Compositions of morphisms and traces} 
\label{compose}

Morphisms and traces on a CohFT algebra admit a notion of composition,
which generalizes the usual homotopy notions of composition in the
\ainfty setting.  The definition of $2$-morphism of a composition of a
trace with a morphism depends on a moduli space of {\em scaled
  parametrized curves} which combines features of the complexified
multiplihedron and cyclohedron.  Let $M_{n,1}(C)$ denote the space of
$n$-marked $1$-scaled curves with underlying curve $C$; we do not
quotient by automorphisms of $C$.  The space $M_{n,1}(C)$ admits a
compactification $\ovl{M}_{n,1}(C)$ by allowing {\em stable scaled
  curves} allowing bubbles with zero area form or allowing the area
form on $C$ to degenerate to zero and {\em affine scaled curves} to
develop as bubbles.  Recall that any nodal map $u: \hat{C} \to C$ of
class $[C]$ has a {\em relative dualizing sheaf} given as the tensor
product of the dualizing sheaf $T^\dual \hat{C}$ for $\hat{C}$ and the
inverse of the pull-back of the cotangent bundle $T^\dual C$ to $C$:
$$T_u^\dual := T^\dual \hat{C} \otimes ( u^* T^\dual C)^{-1} ;$$
(More detail is given below in \cite[Example 4.2]{qk2}.)  
We denote by $ \P(T_u^\dual \oplus \C) $ the associated bundle with
projective line fibers.

\begin{definition}  {\rm (Nodal Scaled Marked Curves)}  
\label{scaledcurves} 
Let $u: \hat{C} \to C$ be a map of class $[C]$.  A {\em scaling form}
is a section $\lambda: \hat{C} \to \P(T_u^\dual \oplus \C) $ such that
on any connected component $\hat{C}'$ of $\hat{C} \ssm C_0$ (that is,
bubble tree attached to the principal component) the pair $(\hat{C}',
\lambda | \hat{C}')$ is an affine scaled curve (if the scaling
$\lambda$ is infinite on $C_0$) or has zero scaling, otherwise.
A {\em nodal scaled curve} parametrized by $C$ is a map $\hat{C} \to
C$ equipped with a scaling form.  An {\em isomorphism} of nodal scaled
parametrized curves $(\hat{C}_j, u_j, \omega_j), j = 0,1$, is an
isomorphism $\psi: \hat{C}_0 \to \hat{C}_1$ such that
\begin{enumerate}
\item {\rm (Scalings are intertwined)} $\psi^* \lambda_1 = \lambda_0$.

\item {\rm (Markings are intertwined)} $\psi(z_{0,i}) = z_{1,i}, i
  =1,\ldots, n$;
\item {\rm (Parametrization is intertwined)} $\psi \circ u_0 = u_1$.
\end{enumerate}  
A nodal scaled parametrized curve is {\em stable} iff it has no
infinitesimal automorphisms, that is, each irreducible non-principal
component has at least three special points or a non-degenerate
scaling and two special points.  The {\em combinatorial type} of a
nodal scaled parametrized curve is a colored tree $\Gamma =
(\Ve(\Gamma),\Edge(\Gamma))$ with finite resp. semiinfinite edges
$\Edge_{< \infty}(\Gamma)$ resp. $\Edge_\infty(\Gamma)$ obtained by
replacing every irreducible component by a vertex and every node or
marking with an edge, equipped with a {\em root vertex} $v_0 \in
\Ve(\Gamma)$ corresponding to the principal component and a partition
$$ \Ve(\Gamma) = \Ve^0(\Gamma) \cup \Ve^1(\Gamma) \cup
\Ve^\infty(\Gamma) $$
corresponding to the irreducible components with zero resp. finite
resp. infinite scalings.  A rooted colored tree is {\em stable} if it
corresponds to a stable scaled curve, that is, every non-root vertex
in $\Ve^0(\Gamma)$ or $\Ve^\infty(\Gamma)$ resp. $\Ve^1(\Gamma)$ has
at least $3$ resp. $2$ incident edges.
\end{definition} \vskip .1in 

In other words, a stable scaled curve is a copy of the curve $C$ with
a section of the trivial bundle (necessarily constant) and a
collection of stable curves attached (if the scaling on the principal
component is zero or finite) or a curve with infinite scaling and a
collection of stable scaled affine lines attached ( if the scaling on
the principal component is infinite).

Let $\ovl{M}_{n,1}(C)$ denote the moduli space of isomorphism classes
of $n$-marked, scaled curves with principal component $C$, and
$M_{n,1,\Gamma}(C)$ the subset of combinatorial type $\Gamma$ so that
$$ \ovl{M}_{n,1}(C) = \cup_\Gamma M_{n,1,\Gamma}(C) .$$
The topology on $\ovl{M}_{n,1}(C)$ is similar to that for
$\ovl{M}_{n,1}(\bA)$ and is compact and Hausdorff.  The local structure
of $\ovl{M}_{n,1}(C)$ near the stratum $M_{n,1,\Gamma}(C)$ of nodal
lines with combinatorial type $\Gamma$ may be described as follows.
As in \eqref{balanced}, for any colored rooted tree $\Gamma$, a
labelling 
$$\gamma: \Edge_{< \infty}(\Gamma) \to \C $$ 
is {\em balanced} iff
\begin{equation} \label{balanced2}
\prod_{e \in P_{vw}} \gamma(e)^{\pm 1} = 1 \end{equation}
where $v,w$ range over elements of $\Ve^1(\Gamma)$ and $P_{vw}$ is the
unique non-self-crossing path from $v$ to $w$, and the sign in the
exponent is positive if $e$ is pointing towards the root vertex and
negative otherwise.  Let $Z_\Gamma$ denote the set of balanced
labellings.  As in \cite{mau:mult} for each rooted tree $\Gamma$ there
exists a tubular neighborhood of the form
$$ M_{n,1,\Gamma}(C) \times Z_\Gamma \to \ovl{M}_{n,1}(C) $$
given by removing small neighborhoods of the nodes and gluing together
using identifications depending on the gluing parameters.

\begin{example}   Suppose that $\hat{C}$ consists of a principal component $C_0 \cong C$ with infinite
scaling, two other components with infinite scaling (dark shading),
four components with finite scaling (medium shading), and two
components with zero scaling (light shading) as shown in Figure
\ref{grels}.  The relations on the gluing parameters $\gamma_1 =
\gamma_2, \gamma_3 = \gamma_4, \gamma_1 \gamma_5 = \gamma_3 \gamma_6$
imply that the curve obtained with gluing with non-zero gluing
parameters is equipped with the area form $\gamma_1 \gamma_5 \omega_C
= \gamma_2 \gamma_5 \omega_C = \gamma_3 \gamma_6 \omega_C = \gamma_4
\gamma_6 \omega_C$.
\end{example} 

\begin{figure}[ht]
\begin{center} 
\begin{picture}(0,0)%
\includegraphics{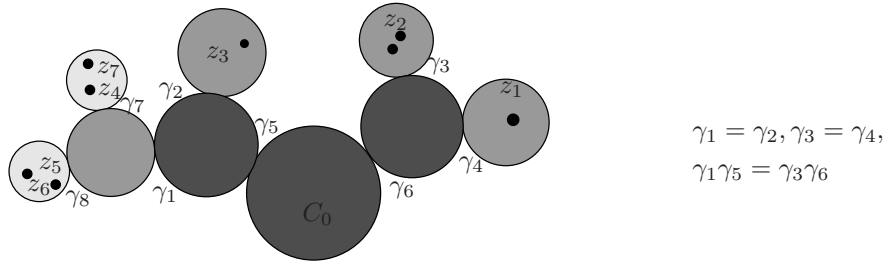}%
\end{picture}%
\setlength{\unitlength}{4144sp}%
\begingroup\makeatletter\ifx\SetFigFont\undefined%
\gdef\SetFigFont#1#2#3#4#5{%
  \reset@font\fontsize{#1}{#2pt}%
  \fontfamily{#3}\fontseries{#4}\fontshape{#5}%
  \selectfont}%
\fi\endgroup%
\begin{picture}(4110,1553)(1899,-3608)
\put(4840,-2624){\makebox(0,0)[lb]{{{{$z_1$}%
}}}}
\put(4146,-2221){\makebox(0,0)[lb]{{{{$z_2$}%
}}}}
\put(2763,-3261){\makebox(0,0)[lb]{{{{{$\gamma_1$}%
}}}}}
\put(2801,-2636){\makebox(0,0)[lb]{{{{{$\gamma_2$}%
}}}}}
\put(4400,-2483){\makebox(0,0)[lb]{{{{{$\gamma_3$}%
}}}}}
\put(4595,-3082){\makebox(0,0)[lb]{{{{{$\gamma_4$}%
}}}}}
\put(3374,-2849){\makebox(0,0)[lb]{{{{{$\gamma_5$}%
}}}}}
\put(4182,-3231){\makebox(0,0)[lb]{{{{{$\gamma_6$}%
}}}}}
\put(3086,-2416){\makebox(0,0)[lb]{{{{$z_3$}%
}}}}
\put(5994,-2900){\makebox(0,0)[lb]{{{{{$\gamma_1 = \gamma_2, \gamma_3 = \gamma_4$, }%
}}}}}
\put(5994,-3125){\makebox(0,0)[lb]{{{{{$\gamma_1 \gamma_5 = \gamma_3 \gamma_6$}%
}}}}}
\put(2084,-3076){\makebox(0,0)[lb]{{{{$z_5$}%
}}}}
\put(3661,-3391){\makebox(0,0)[lb]{{{{{$C_0$}%
}}}}}
\put(2572,-2723){\makebox(0,0)[lb]{{{{{$\gamma_7$}%
}}}}}
\put(2243,-3282){\makebox(0,0)[lb]{{{{{$\gamma_8$}%
}}}}}
\put(2015,-3208){\makebox(0,0)[lb]{{{{$z_6$}%
}}}}
\put(2438,-2637){\makebox(0,0)[lb]{{{{$z_4$}%
}}}}
\put(2430,-2504){\makebox(0,0)[lb]{{{{$z_7$}%
}}}}
\end{picture}%
\end{center} 
\caption{Gluing relations on a marked scaled curve} 
\label{grels}
\end{figure}

\begin{proposition}  
The boundary of $\ovl{M}_{n,1}(C)$ is the union of the following sets
(which will be {\em divisors} once the algebraic structure on
$\ovl{M}_{n,1}(C)$ is introduced)
\begin{enumerate}
\item {\rm (Bubbling points)} For any subset $I \subset \{1, \ldots, n \}$
  of order at least two we have a subspace
$$ \iota_I: D_I \to \ovl{M}_{n,1}(C) $$
and an isomorphism 
$$ \varphi_I: D_I \to \ovl{M}_{0,|I|+1} \times \ovl{M}_{n- |I| + 1,1}(C) $$
corresponding to the formation of a bubble with markings $z_i, i \in
I$ with zero scaling.  
\item {\rm (Blowing up scaling)} For any unordered partition $[I_1,
  \ldots,I_r]$ of $\{1,\ldots, n\}$ we have a subspace
$$ \iota_{[I_1,\ldots,I_r]} :D_{[I_1,\ldots,I_r]} 
\to \ovl{M}_{n,1}(C) $$
and an fibration forgetting all but the infinitely-scaled component
$$  D_{[I_1,\ldots,I_r]}
\to \ovl{M}_{r}(C)$$
whose fiber isomorphic to $\prod_{j=1}^r \ovl{M}_{|I_j|,1}(\bA)$
describes bubbles with non-degenerate area form containing the
markings.\footnote{This sentence corrected from the published
  version.}
\item {\rm (Fixing a scaling)} For any $\rho \in \C$, in particular
  for $\rho = 0$, there is an inclusion
$$\iota_\rho: \ovl{M}_{n}(C) \to \ovl{M}_{n,1}(C) $$
by choosing any scaling ${\rho} \omega_C$.
\end{enumerate} 
\end{proposition} 

\begin{proof}   The description of boundary subspaces is immediate from 
the tubular neighborhood description of each stratum and a dimension
count, which is the same as in Remark \ref{dimcount}. 
\end{proof} 

The adiabatic limit Theorem \ref{largearea} will be deduced from the
following divisor class relation:

\begin{proposition}  \label{rhoequiv} {\rm (The basic divisor class relation in the moduli of stable
scaled curves)} The homology class of $\iota_\rho(\ovl{M}_n(C))$ is
  equal to that of the union of classes of $D_{[I_1,\ldots,I_r]}$ over
  unordered partitions.
\end{proposition} 

\begin{proof}  
The equivalence in homology induced by the map $\rho: \ovl{M}_{n,1}(C)
\to \P$ equates the homology classes of $\rho^{-1}(0) \cong
\ovl{M}_n(C)$ with $\rho^{-1}(\infty)= \cup_{r,[I_1,\ldots,I_r]}
D_{[I_1,\ldots,I_r]}$, as one can check that the multiplicity of each
divisor on the right hand side is $1$, using the fact that linear
equivalence of divisors implies homology equivalence.
\end{proof}

Suppose that $V,W$ are (even, genus zero) CohFT algebras, with
structure maps
$$\mu_V^n: V^n \times H(\ovl{M}_{0,n+1}) \to V, \ \ \ \mu_W^n: W^n
\times H(\ovl{M}_{0,n+1}) \to W .$$
Let $\phi^n: V^n \times H(\ovl{M}_{n,1}(\bA)) \to V$ be a morphism of
CohFT algebras, and $\tau_V,\tau_W$ traces on $V,W$ respectively.

\begin{definition} \label{2morphism}   {\rm ($2$-morphisms for compositions of traces
and morphisms)}  
A {\em $2$-morphism from $\phi \circ \tau_W$ to $\tau_V$} is a
collection of maps
$$ \psi: V^n \times H(\ovl{M}_{n,1}(C)) \to W, \quad n \ge 0  $$
such that
\begin{enumerate} 
\item {\rm (Fixing Scaling)} if $\gamma \in H^2(\ovl{M}_{n,1}(C))$
  is the dual class to $\iota_\rho(\ovl{M}_n(C))$ then
$$
  \psi^n( \alpha_1,\ldots, \alpha_n; \beta \cup  \gamma ) =
  \tau_V^n( \alpha_1,\ldots, \alpha_n; \iota_\rho^* \beta) ;
$$
\item {\rm (Bubbling points)}  if $\gamma_{I} \subset H^2(\ovl{M}_{n,1}(C))$
is the dual class to the divisor $D_{I}$ corresponding to bubbling 
off markings $z_i, i \in I$, then 
$$ \psi^n( \alpha_1,\ldots, \alpha_n; \beta \cup \gamma_{I} ) =
\psi^{n - |I| + 1}( \alpha_j,j \notin I, \mu^{|I|} (\alpha_i, i \in I,
\cdot ) ; \cdot) (\iota_{I,1}^* \beta) ;$$
\item {\rm (Blowing up scaling)} if $D = \sum n_{[I_1,\ldots,I_r]} D_{[I_1,\ldots,I_r]}$ is
  a boundary divisor with dual class $\gamma$ then
\begin{multline} 
\psi^n( \alpha_1,\ldots, \alpha_n; \beta \cup \gamma ) = 
  \sum_{[I_1, \ldots,I_r]} n_{[I_1,\ldots,I_r]}  \tau^r_W( \phi^{|I_1|}(\alpha_i, i \in
  I_1;\cdot),\ldots, \\ \phi^{|I_r|}(\alpha_i, i \in I_r; \cdot); \cdot)
  ( \iota_{[I_1,\ldots,I_r]}^* \beta ) .
\end{multline}
\end{enumerate} 
We write $\tau_V \cong_\psi \tau_W \circ \phi$.  
\end{definition} \vskip .1in 

\begin{lemma} If $\tau_V \cong_\psi \tau_W \circ \phi$ then 
$$ 
 \tau_V( \alpha_1,\ldots, \alpha_n; \iota_\rho^* \beta) = 
  \sum_{r,[I_1, \ldots,I_r]} \tau^r_W( \phi^{|I_1|}(\alpha_i, i \in
  I_1;\cdot),\ldots, \phi^{|I_r|}(\alpha_i, i \in I_r; \cdot); \cdot)
  ( \iota_{[I_1,\ldots,I_r]}^* \beta )  $$
\end{lemma} 

\begin{proof}   The proof follows 
by combining (a) and (c) in the definition.  These are applied to
using the equality of the homology class $[D_\rho]$ of the divisor
$D_\rho$ corresponding to fixed scaling and the divisor at infinite
scaling $\sum [D_{[I_1,\ldots,I_r]}]$.
\end{proof}



 \section{Symplectic vortices}
\label{vor}

In physics, a {\em vortex} refers to a stable solution of classical
field equations which has finite energy in two spatial dimensions, see
for example Preskill \cite{pres:vm} and, for a more mathematical
treatment, Jaffe-Taubes \cite{jt:vm} who classified vortices for
scalar fields.  In mathematics, vortices often refer to pairs of a
connection and section of a line bundle satisfying an equation
involving the curvature and a quadratic function of the section, see
for example Bradlow \cite{brad:vor}.  Symplectic vortices are vortices
in which the ``field'' takes values in a symplectic manifold with
Hamiltonian group action.  In this section we review the symplectic
approach to gauged Gromov-Witten invariants, also known as symplectic
vortex invariants or Hamiltonian gauged Gromov-Witten invariants, as
introduced by Mundet and Salamon, see \cite{ci:symvortex},
\cite{mun:ham}.  If the moduli spaces of symplectic vortices are
smooth, then integration over them defines the required invariants and
the proofs of the Theorems \ref{large}, \ref{largearea} are immediate.
Of course, the moduli spaces are not smooth, or of expected dimension
in general, and to define the needed virtual fundamental cycles we
pass to algebraic geometry, starting in the following section.

\subsection{Gauged holomorphic maps}

In this section we review the construction of the moduli space of
symplectic vortices as the symplectic quotient of the action of the
group of gauge transformations on the space of gauged holomorphic
maps; this generalizes the construction of the space of flat
connections as the symplectic quotient of the action of the group of
gauge transformations on the space of connections on a bundle.
Unfortunately since the space of gauged holomorphic maps is in general
singular, this construction is rather formal and serves only for
motivation for what follows.

We begin with notation for equivariant cohomology.  Let $K$ be a
compact group with Lie algebra $\k$.  We assume that $\k$ is equipped
with a $K$-invariant metric, inducing an identification $\k^\dual \to
\k$.  We denote by $EK \to BK$ a universal $K$-bundle, unique up to
homotopy equivalence.  For any $K$-space $X$, we denote by $X_K := X
\times_K EK$ the homotopy quotient and by $H^K(X) := H(X_K)$ the
equivariant cohomology.

Next we introduce connections and curvature. Let $P \to C$ be a
principal $K$-bundle over a compact surface $C$, and $\psi: C \to BK$
a classifying map for $P$.  Denote by
$$ \Omega(P,\k)^K = \{ \theta \in \Omega(P,\k) \ |\  (k^{-1})^* \theta = \Ad(k) \theta \} $$
the space of $K$-invariant forms and by
$$\A(P) = \{ \theta \in \Omega^1(P,\k)^K \ |
\ \theta_{\pi(p)}(\xi_P(p)) = \xi, \forall \xi \in \k,  p \in P \} $$
the space of (principal) connections on $P$, where $\xi_P(p) = \ddt
|_{t =0} p \exp( t \xi)$ is the generating vector field at $p$.  The
space $\A(P)$ is an affine space with a free transitive action of the
space $\Omega^1(C,P(\k))$ of one-forms with values in the {\em adjoint
  bundle} $ P(\k) = P \times_K \k .$ Let
$$ \A(P) \to \Omega^2(C,P(\k)), \quad A \mapsto F_A $$
denote the curvature map.  

For Hamiltonian actions we introduce a notation of a gauged map, which
will mean a pair of a connection and a section of the associated
bundle.  Let $X$ be a Hamiltonian $K$-manifold with symplectic form
$\omega$ and moment map $\Phi:X \to \k^\dual$.  By our convention this
means that for all $\xi \in \k$, we have $\omega(\xi_X, \cdot) = - \d
( \Phi, \xi)$ where $\xi_X(x) = \ddt |_{t =0 } \exp( - t \xi) x$ is
the generating vector field for $\xi$.  Recall that $P(X)$ is the
associated $X$-bundle.  Sections $u: C \to P(X)$ are in one-to-one
correspondence with lifts $u_K$ of $\psi$ to $X_K$.  Given a section
$u: C \to P(X)$, the homology class $[u]$ is defined to be the
homology class $[u] := u_{K,*} [C] \in H_2^K(X,\Z).$ The equivariant
symplectic form $\omega_K \in \Omega^2_K(X)$ pulls back to
$\Omega^2_K(P \times X)$ and descends to a closed, fiber-wise
symplectic two-form $\omega_{P(X),A} \in \Omega^2(P(X))$ depending on
the choice of connection $A$.  Its cohomology class $[\omega_{P(X)}]
\in H^2(P(X))$ is independent of the choice of connection.  The
\emph{equivariant symplectic area} of $u$ is
$$ D(u) = ([u], [\omega_K]) = ([C], u^* [\omega_{P(X)}] ) $$ 
where $( \ , \ )$ denotes the pairing between homology and cohomology.
More concretely, we have $ D(u) = \int_C u^* \omega_{P(X),A} $
independently of the connection $A$.  We denote by
$P(\Phi): P(X) \to P(\k^\dual) \cong P(\k) $
the map induced by $\Phi$.  A {\em gauged map} from $C$ to $X$ is a
datum $(P,A,u)$ where $A \in \A(P)$ and $u: C \to P(X)$ is a section.
Given a metric on $C$, we denote by 
$$ *: \Omega^\bullet(C) \to \Omega^{2 - \bullet}(C) $$
the associated Hodge star.  The {\em energy} of a gauged map $(A,u)$
is given by
\begin{equation} \label{energy} E(A,u) = \hh \int_C  * \left( \Vert \d_A u \Vert^2 + \Vert F_A\Vert^2 + \Vert u^*
 P(\Phi)\Vert^2 \right) .\end{equation}

Finally we introduce gauged holomorphic maps and symplectic vortices.
Suppose the $C$ is a complex curve.
%
Denote by $\J(X)$ the space of almost complex structures on $X$
  compatible with $\omega$. 
The action of $K$ induces an action on $\J(X)$, and we denote by
  $\J(X)^K$ the invariant subspace.  
Any connection $A \in \A(P)$ induces a map of spaces of almost
  complex structures
$$ \J(X)^K \to \J(P(X)), \ \ J \mapsto J_A$$
by combining the almost complex structure on $X$ and $C$ using
the splitting defined by the connection.  
%
Let $\Gamma(C,P(X))$ denote the space of sections of $P(X)$. 
We denote by
$$ \olp_A : \Gamma(C,P(X)) \to \bigcup_{u \in \Gamma(C,P(X))}
\Omega^{0,1}(C,u^*T^{\on{vert}} P(X)) $$
the Cauchy-Riemann operator defined by $J_A$.  

\begin{definition} {\rm (Gauged holomorphic maps)} 
Suppose that $X$ is a Hamiltonian $K$-manifold equipped with an
invariant almost complex structure $J$, $C$ is a complex curve, and $P
\to C$ is a principal $K$-bundle.  A {\em gauged holomorphic map} for
$P$ is a pair $(A,u)$ satisfying $\olp_A u = 0 $.\end{definition} \vskip .1in

Let $\H(P,X)$ be the space of gauged holomorphic maps for $P$:
$$ \H(P,X) = \{ (A,u) \in \A(P) \times \Gamma(C,P(X)), \ \ \olp_A u
= 0 \} .$$
The energy and symplectic area are related by
a generalization of the familiar energy-area relation for
pseudoholomorphic maps in \cite{ms:jh}:

\begin{proposition} {\rm (Energy-area relation, \cite[Proposition 3.1]{ci:symvortex})}
Suppose that $X$ is a Hamiltonian $K$-manifold equipped with an
invariant almost complex structure $J$, $C$ is a compact complex
curve, and $P \to C$ is a principal $K$-bundle.  Let $(A,u)$ be a
gauged map from $C$ to $X$ with bundle $P$.  Let $\omega_C$ be the
area form determined by a choice of metric on $C$. The energy and
equivariant symplectic area are related by
\begin{equation}  \label{energyaction} 
E(A,u) = D(u) + \int_C * \Vert \olp_A u \Vert^2 + * \hh \Vert F_A + u^* P(\Phi) \omega_C
\Vert^2.
\end{equation}
\end{proposition} 

\begin{definition}  \label{vortexdef}
A gauged map $(A,u) \in \H(P,X)$ is a {\em symplectic vortex} if it
satisfies the
$$ {\rm (Vortex\ Equation)} \quad F_{A,u} := F_A +  u^* P(\Phi)
\Vol_{C} = 0 .$$
An {\em isomorphism} of symplectic vortices $(A_j,u_j), j =0,1$ with
bundle $P$ is a gauge transformation $\phi: P \to P$ such that $\phi^*
A_1 = A_0$ and $ \phi(X) \circ u_0 = u_1$, where $\phi(X): P(X) \to
P(X)$ is the fiber-bundle-automorphism induced by $\phi$.  An {\em
  $n$-marked} symplectic vortex is a vortex $(A,u)$ together with an
$n$-tuple $\ul{z} = (z_1,\ldots, z_n)$ of distinct points in $C$.  A
{\em framed vortex} is a collection $(A,u,\ul{z},\ul{\phi})$, where
$(A,u,\ul{z})$ is a marked vortex and $\ul{\phi} =
(\phi_1,\ldots,\phi_n)$ is an $n$-tuple where each $\phi_j: P_{z_j}
\to K$ is a $K$-equivariant isomorphism, that is, a trivialization of
the fiber. \end{definition}

\vskip .1in
Let $M_n(P,X)$ denote the moduli space of isomorphism classes of
$n$-marked vortices and $ M_n^K(C,X)$ the union over
isomorphism classes of bundles $P \to C$.  The moduli space
$M_n^K(C,X)$ is homeomorphic to the product $M^K(C,X) \times
M_n(C) $ where $M_n(C)$ denotes the configuration space of $n$-tuples
of distinct points in $C$.

\subsection{Nodal symplectic vortices} 

The bubbling phenomenon for holomorphic curves also occurs in the case
of symplectic vortices and prevents compactness of the moduli spaces.
However, once one incorporates bubbles and fixes the homology class
the moduli spaces become compact, as we now explain following Salamon
et al \cite{ci:symvortex} who proved compactness of the moduli space
of vortices of fixed homology class in the case that $X$ has no
holomorphic spheres, and Mundet \cite{mun:ham} and Ott
\cite{ott:remov} who compactify the moduli space of vortices by
allowing bubbling in the fibers of $P(X)$.

\begin{definition} \label{nodal} {\rm (Nodal vortices)} 
Suppose that $X$ is a Hamiltonian $K$-manifold equipped with an
invariant almost complex structure $J$ and $C$ is a connected smooth
projective curve.  A {\em nodal gauged $n$-marked map} from $C$ to $X$
with underlying bundle $P$ consists of a datum
$(P,A,\hat{C},{u},\ul{z})$ where
\begin{enumerate}
\item {\rm (Bundle and Connection)} $P \to C$ is a $K$-bundle and $A$ is a
  connection on $P$;
\item {\rm (Nodal section)} $u: \hat{C} \to P(X)$ is a map from nodal
  curve $\hat{C}$ to $P(X)$ of base degree one, that is, $\pi \circ u:
  \hat{C} \to C$ is a map of class $[C]$;
\item {\rm (Markings)} an $n$-tuple $\ul{z} = (z_1,\ldots,z_n) \in
  \hat{C}^n$ of distinct, smooth points of $\hat{C}$.
\end{enumerate}
An {\em isomorphism} of nodal gauged maps
$(\hat{C}',A',{u}',\ul{z}'),(\hat{C}'',A'',{u}'',\ul{z}'')$ with
bundles $P',P''$ consists of
\begin{enumerate} 
\item {\rm (Domain automorphism)} an automorphism of the domain $\psi: \hat{C}' \to \hat{C}''$,
inducing the identity on $C$, and 
\item {\rm (Bundle automorphism)}  a bundle isomorphism $\phi: P' \to P''$, inducing the identity
  on $C$,
\end{enumerate} 
intertwining the connections and maps and exchanging the markings,
that is,
\begin{enumerate} 
\item {\rm (Isomorphism of connections)} $\phi^* A'' = A'$, 
\item {\rm (Isomorphism of maps)} ${u}'' = \phi(X) \circ {u}' \circ
  \psi^{-1} $ where $\phi(X): P(X) \to P'(X)$ is the map induced by
  $\phi$ and
\item {\rm (Isomorphism of markings)}  $\psi(z_i') = z_i'', i =1,\ldots, n$.  
\end{enumerate} 
A {\em nodal vortex} is a nodal gauged map such that the principal
component is a vortex and there are no automorphisms with trivial
bundle automorphism (gauge transformation).  A nodal vortex is {\em
  stable} if the map $u: \hat{C} \to P(X)$ is a stable map from the
marked curve $(\hat{C},\ul{z})$ and the pair $(A,u)$ has finite
automorphism group under the action of gauge transformations.  Note
that there is no condition on the number of special points on the
principal component. In particular, constant gauged maps with no
markings can be stable.  For any map ${u}:\hat{C} \to P(X)$, the {\em
  homology class} $[u] \in H_2^K(X,\Z)$ of ${u}$ is the push-forward
of $[\hat{C}]$ under a map $u_K: \hat{C} \to X_K$ obtained from a
classifying map for $P$.
\end{definition} \vskip .1in  

The following extends the notion of convergence to the case of nodal
marked vortices.

\begin{definition}  \label{vortexconverge}
{\rm (Convergence of nodal vortices)} Suppose that $X$ is a
Hamiltonian $K$-manifold equipped with an invariant almost complex
structure $J$ and $C$ is a connected smooth projective curve.  Suppose
that $(P_\nu,A_\nu,\hat{C}_\nu,u_\nu,\ul{z}_\nu)$ is a sequence of
marked nodal vortices on $C$ with values in $X$, and
$(P,A,\hat{C},{u},\ul{z})$ is a nodal vortex with values in $X$.  We
say that $[(P_\nu,A_\nu,\hat{C}_\nu,A_\nu,u_\nu,\ul{z}_\nu)]$ {\em
  converges} to $[(P,A,\hat{C},A,\ul{z})]$ if after a sequence of
bundle isomorphism $\phi_\nu: P_\nu \to P$ the connections $A_\nu$
converge to $A_\infty$ in the $C^0$ topology and $[(\hat{C}_\nu,
u_\nu,\ul{z}_\nu)]$ Gromov converges to $[(\hat{C},u,\ul{z})]$.
\end{definition} \vskip .1in

See Ott \cite{ott:remov} for more details on the definition of
convergence.  The definition of convergence implies in particular that
the curvature $F_{A_\nu}$ converges to $F_{A_\infty}$ in $L^2$, but
$A_\nu$ is not required to (and does not) converge to $A$ uniformly in
all derivatives.  Recall from \cite{ci:symvortex} a condition which
guarantees compactness of moduli spaces of symplectic vortices with
non-compact target:

\begin{definition}  A Hamiltonian
$K$-manifold $X$ with moment map $\Phi: X \to \k^\dual$ is {\em convex
    at infinity} if there exists $f \in C^\infty(X)^G$ such that
\begin{equation} \label{convex}  
{\rm (Convexity)} \ \ 
\langle \nabla_v \nabla f(x), v
\rangle + 
\langle \nabla_{Jv} \nabla f(x), Jv
\rangle
 \ge 0, \quad \d f(x) J_x \Phi(x)_X (x) \ge 0 \end{equation}
for every $x \in X$ and $v \in T_x X$ outside of a compact subset of
$X$.  
\end{definition} \vskip .1in 

For example, if $X$ is a vector space with a linear Hamiltonian action
of $K$ with {\em proper} moment map $\Phi$ then $X$ is convex.  The
following is proved by Mundet \cite{mun:ham} and Ott \cite{ott:remov}.

\begin{theorem}
\label{compactnessforvortices} {\rm (Sequential compactness
for vortices with compact domain)} Suppose that $X$ is Hamiltonian
$K$-manifold equipped with a proper moment map convex at infinity and
an invariant almost complex structure $J$, $C$ is a connected smooth
compact complex curve.  Any sequence of nodal symplectic vortices with
bounded energy has a convergent subsequence.
\end{theorem} 

Convergence for sequences with bounded first derivative is proved as
follows: Suppose that $(A_\nu,u_\nu)$ is a sequence of symplectic
vortices on a bundle $P$ with smooth domain with the property that $
c_\nu := \sup \Vert d_{A_\nu} u_\nu \Vert = \Vert\d_{A_\nu}
u_\nu(z_\nu)\Vert$ is bounded.  By Uhlenbeck compactness, after gauge
transformation, we may assume that $A_\nu$ converges weakly in the
Sobolev $W^{1,p}$ topology and strongly in $C^0$ to a limit
$A_\infty$.  After putting $A_\nu$ in Coulomb gauge with respect to
$A_\infty$, elliptic regularity for vortices (see
\cite[p.20]{ci:symvortex}) implies that $(u_\nu,A_\nu)$ has a
$C^l$-convergent subsequence of any $l$.  

More generally in the case of unbounded first derivative one has a
bubbling analysis similar to that for pseudoholomorphic curves
\cite{ci:symvortex}, \cite{ott:remov}: If $\Vert\d_{A_\nu}
u_\nu(z_\nu)\Vert \to \infty $ and $z_\nu \to z$ we say that $z$ is a
{\em bubble point} for the sequence $(A_\nu,u_\nu)$.
\begin{theorem} \label{annlem} 
{\rm (Bubbling analysis for vortices with compact domain)  }  Let $X$ be a 
Hamiltonian $K$-manifold with proper moment map convex at infinity.  
\begin{enumerate} 
\item {\rm (Removal of Singularities)} \cite[Theorem 1.1]{ott:remov} Any
  finite energy vortex $(A,u)$ on the punctured disk $ D - \{ 0 \}$
  extends to a vortex on $D$.
\item {\rm (Energy Quantization)} \cite[Lemma 4.2]{ott:remov} There
  exists an $E_0 > 0 $ such that for any bubble point of a sequence
  $(A_\nu,u_\nu)$, one has $$\lim_{\eps \to 0} \lim_{\nu \to \infty}
  E(A_\nu,u_\nu | B_{\eps}(z)) > E_0 .$$
\item {\rm (Annulus Lemma)} \cite[Section 5.2]{ott:remov} 
There exists
  a constant $\eps > 0$ such that if $(A,u)$ is a vortex on 
$  \AA(r,R) := B_R(0) - B_r(0) $
then for every $\mu < 1$ there exist constants $R_0, \delta_0,C > 0$
such that if $R > R_0$ and $E(A,u) < \delta_0$, then for $\log(2) \leq
T \leq \log(\sqrt{R/r})$ the restriction of $(A,u)$ to $\AA(e^T r,
e^{-T} R)$ satisfies
$$E ( (A, u) |_{\AA(e^T r, e^{-T}R}) < C e^{- 2 \mu T} E(A,u) .$$
\item {\rm (Mean Value Inequality)} \cite[Corollary 2.2]{ott:remov}
  Let $r> 0$ and let $\lambda(s,t) \d s \d t $ be an area form on the
  ball $B_r(0)$ of radius $r$ around $0$ in $\C$.  There exist
  constants $C,\delta > 0$ such that if $(A,u)$ is a vortex on
  $B_r(0)$ with $E(A,u | B_r(0)) < \delta$ then
$$ (1/2) \Vert\d_A u (0)\Vert^2 + \Vert \Phi(u(0)) \Vert^2 < (C/r^2)
  E( (A,u ) |_{ B_r(0)}) .$$
We remark that an examination of the proof in \cite{ott:remov} shows
that the constants $C,\delta$ depend only on bounds on $\lambda$ and
its first and second derivatives.
\end{enumerate} 
\end{theorem}

\begin{remark}  Given these ingredients the proof of compactness goes as follows.  For
each {\em bubbling sequence} $z_\nu$ with $ \Vert \d_{A_\nu}
u_\nu(z_\nu) \Vert \to \infty $ one constructs by {\em soft rescaling}
a sequence of $J_{A^1_\nu}$-holomorphic maps $u^1_\nu$ on balls of
increasing radius, with the property that the limit $A^1_\nu$ is zero
and $u^1_\nu$ is an ordinary holomorphic map from $\C$ to $X$.  Note
that one does not have $C^2$ convergence of the sequence
$J_{A^1_\nu}$; however, \cite[Appendix]{ott:remov} shows that $C^0$
convergence of a sequence of almost complex structures is sufficient
as long as one has a version of the mean value inequality, which in
this case follows from the vortex equations.  The limiting
configuration is then constructed by induction.\end{remark} \vskip .1in

Denote by $\ovl{M}_n^K(C,X)$ the set of isomorphism classes of nodal
vortices with connected domain.  The {\em combinatorial type} of a
nodal vortex is a rooted tree $\Gamma$ with root vertex corresponding
to the principal component, equipped with a labelling of vertices by
elements of $H_2^K(X,\Z)$.  For any tree $\Gamma$ and homology class
$d \in H_2^K(X,\Z)$ we denote by $M^K_{n,\Gamma}(C,X,d)$ the space of
isomorphism classes of vortices of combinatorial type $\Gamma$, so
that
$$ \ovl{M}_n^K(C,X) = \bigcup_{\Gamma} M^K_{n,\Gamma}(C,X,d) $$
as $\Gamma$ ranges over connected combinatorial types.  We say that a
subset $S$ of $\ovl{M}_n^K(C,X)$ is {\em closed} if any convergent
sequence in $S$ has limit point in $S$, and {\em open} if its
complement is closed.  The open sets form a topology for which any
convergent sequence is convergent, and any convergent sequence has a
unique limit, by arguments similar to \cite[Lemma 5.6.5]{ms:jh}.
Namely local ``distance functions'' can be defined by combining the
distance functions of \cite{ms:jh} with the $L^2$-metric on the space
of connections.  In the case $n = 0$, given a constant $\eps > 0$ and
stable vortex $(P,A_0,\hat{C}_0,{u}_0)$ and another stable vortex
$(P,A_1,\hat{C}_1,{u}_1)$ with the same underlying bundle $P$ define
\begin{multline} \label{distfns}
 \dist_\eps([(P,A_0,\hat{C}_0,{u}_0)],[(P,A_1,\hat{C}_1,{u}_1)]) =
 \\ \inf_{k \in \K(P)} \Vert k \cdot A_1 - A_0 \Vert_{L^2} + \dist_\eps^0(
     [(\hat{C}_0,u_0)], [(\hat{C}_1,k \cdot u_1)]) \end{multline}
where $\dist_\eps^0$ is the distance on isomorphism classes of stable
maps defined on \cite[p. 134]{ms:jh}, but using the Yang-Mills-Higgs
energy on a small ball around each node.  It follows from the results
in Ott \cite{ott:remov} and elliptic regularity for vortices that for
$\eps $ sufficiently small, a sequence
$[(P,A_\nu,\hat{C}_\nu,{u}_\nu)]$ converges to
$[(P,A_0,\hat{C}_0,{u}_0)]$ if and only if
$$
\dist_\eps([(P,A_0,\hat{C}_0,{u}_0)],[(P,A_\nu,\hat{C}_\nu,{u}_\nu)])
\to 0 $$
and the remainder of the proof is similar to that in \cite{ms:jh}.
The sequential compactness Theorem \ref{compactnessforvortices}
implies:

\begin{theorem}  {\rm (Properness of the moduli space of symplectic vortices)}
Suppose that $X$ is a Hamiltonian $K$-manifold equipped with a proper
moment map convex at infinity and an invariant almost complex
structure $J$ and $C$ is a connected smooth projective curve.  Then
$\ovl{M}_n^K(C,X)$ is Hausdorff and the energy map $E: \ovl{M}_n^K(C,X)
\to [0,\infty)$ is proper.
\end{theorem} 

\subsection{Large area limit}

In the large area limit the vortices are related to holomorphic maps
to the (possibly orbifold) symplectic quotient, as pointed out by
Gaio-Salamon \cite{ga:gw}.  The limiting process involves various
kinds of bubbling which one hopes to incorporate into a description of
the relationship.  

First recall the notion of symplectic quotient of $X$ by $K$ as
introduced by Mayer and Marsden-Weinstein: Suppose that $X$ is a
Hamiltonian $K$-manifold equipped with a proper moment map $\Phi: X
\to \k^\dual$.  Let
$$ X \qu K := \Phi^{-1}(0)/ K $$
denote the symplectic quotient.  If $K$ acts freely on $\Phi^{-1}(0)$,
then $X \qu K$ has the structure of a smooth manifold of dimension
$\dim(X) - 2 \dim(K)$.  The quotient $X \qu K$ has a unique symplectic
form $\omega_0$ satisfying $i^* \omega = p^* \omega_0$, where $i:
\Phi^{-1}(0) \to X$ and $p: \Phi^{-1}(0) \to X \qu K$ are the
inclusion and projection respectively.  Any invariant almost complex
structure $J$ on $X$ induces an almost complex structure on $X \qu K$.

If $X \subset \P(V)$ is a smooth projectively embedded variety in a
$G$-representation $V$ then $X \qu K$ is canonically homeomorphic to
the geometric invariant theory quotient $X \qu G$ introduced by
Mumford, by a theorem of Kempf-Ness.  The latter is defined as the
quotient of the {\em semistable locus}
$$ X^{\sss} = \{ x \in X | \exists k \in \Z_+, s \in H^0(X,\mO_X(k))^G,
s(x) \neq 0 \} $$
where $\mO_X(k)$ is the $k$-th tensor product of the hyperplane bundle
on $X$.  The quotient $X \qu G$ is the quotient of $X^{\sss}$ by the
{\em orbit-equivalence relation}
$$ x_1 \sim x_2 \iff \ovl{Gx_1}
\cap \ovl{G x_2} \cap X^{\sss} \neq \emptyset .$$ 
A point $x \in X$ is {\em stable} if $Gx$ is closed in $X^{\sss}$ and
the stabilizer $G_x$ is finite.  If stable=semistable in $X$ then the
points of $X \qu G$ are the orbits of $G$ in $X^{\sss}$, that is, two
orbits are equivalent iff they are equal.

Let $C$ be a connected smooth projective curve, and  suppose that $X
\qu K$ is a locally free quotient.  

\begin{definition}  A gauged holomorphic map $(A,u)$ is an
    {\em infinite-area vortex} if it satisfies
\label{limeq} $u^* P(\Phi ) = 0$.  
\end{definition} \vskip .1in  

Let $M^K(C,X)_\infty$ denote the set of gauge-equivalence classes of
infinite-area vortices, and $M(C,X\qu K)$ the set of holomorphic maps
from $C$ to $X \qu K$.

\begin{proposition}   \label{bij}
Suppose that $K$ acts locally freely on $\Phinv(0)$.  Then
there is a bijection from $M^K(C,X)_\infty$ to $M(C,X \qu
K)$.  \end{proposition}

\begin{proof}   See Gaio-Salamon \cite[Section 2]{ga:gw}.  
Given an infinite-area vortex $(A,u)$, let $\ovl{u}: C \to X \qu K$
denote the composition of $u$ with the quotient map $\Phinv(0) \to X
\qu K$.  The equation $\olp_A u = 0$ implies $\olp \ovl{u} = 0$ (since
$TX \to T(X \qu K)$ is holomorphic) and so $\ovl{u} \in M(C, X \qu
K)$. Conversely, given $\ovl{u}: C \to X \qu K$ let $P$ be the
pull-back of $\Phinv(0) \to X \qu K$, equipped with the connection
given by $J TX | \Phinv(0) \cap TX | \Phinv(0) \cong \pi^* T(X \qu
K)$.  The equivariant map $P \to \Phinv(0)$ defines a section $u: C
\to P \times_K X$.  The equation $\olp \ovl{u} = 0$ implies $\olp_A u =
0$, since $J_A$ agrees with $J_{X \qu K}$ on $\pi^* T(X \qu K)$.  If
$(A,u)$ is constructed in this way from $\ovl{u}$ then the
corresponding map to $X \qu K$ is $\ovl{u}$.  To see that the map
$(A,u) \mapsto \ovl{u}$ is injective, suppose that $(A,u)$ and
$(A',u')$ define the same holomorphic map to $X \qu K$.  This means
that there is a gauge transformation $k$ so that $ku = u'$. the
equations $\olp_A u = 0 = \olp_{A'} u'$ and local freeness of the
action imply that $kA = A'$.
\end{proof} 

Gaio-Salamon \cite{ga:gw} 
studied under what conditions a sequence of
symplectic vortices defined with respect to an area form $\rho_\nu
\omega_C$ with $\rho_\nu \to \infty$ converge to a solution of the
limiting equations:

\begin{proposition} \label{sc} {\rm (Sequential compactness for bounded first derivative, \cite{ga:gw} )}  
Suppose that $X$ is a Hamiltonian $K$-manifold with proper moment map
convex at infinity equipped with an invariant almost complex structure
$J$ and $C$ is a connected smooth projective curve.  Suppose that
$(A_\nu,u_\nu)$ is a sequence of symplectic vortices for $\rho_\nu
\omega_C$ with the property that $\sup \Vert\d_{A_\nu} u_\nu\Vert$ is
bounded. Then there exists a subsequence and a sequence of gauge
transformations $k_\nu$ such that $k_\nu(A_\nu,u_\nu)$ converges to an
infinite-area vortex $(A_\infty,u_\infty)$ uniformly in all
derivatives.
\end{proposition} 

\vskip .1in Without a bound on the first derivative, various kinds of
bubbling occur.

\begin{proposition} {\rm (Bubble zoology for the infinite area limit)} 
Suppose that $X$ is a Hamiltonian $K$-manifold with a proper moment
map convex at infinity and an invariant almost complex structure $J$
and $C$ is a connected smooth projective curve.  Suppose that
$(A_\nu,u_\nu)$ is a sequence of vortices for $\rho_\nu \Vol_C,
\rho_\nu \to \infty$ with
$$c_\nu := \sup \Vert \d_{A_\nu} u_\nu \Vert = \Vert \d_{A_\nu} u_\nu
(z_\nu ) \Vert, \quad \eps_\nu := \rho_\nu /c_\nu .$$
Consider the {\em rescaled pair} $ \phi_\nu^* (A_\nu,u_\nu)$ on
$B_{c_\nu}(0)$ where $ \phi_\nu(z) = z_\nu + z/c_\nu $.  Noting that
$\phi_\nu^* (A_\nu,u_\nu)$ has $ \sup \Vert \d_{\phi_\nu^* A_\nu} u_\nu \Vert = 1
$ bounded.  Then after passing to a subsequence one of the following
possibilities occurs:

\begin{enumerate} 
\item {\rm (Sphere Bubble in $X \qu K$)} If $\lim_{\nu \to \infty} \eps_\nu =
  \infty$, then $\phi_\nu^* (A_\nu,u_\nu)$ converges to a solution to
  a solution to \ref{limeq}, that is, is equivalent to a holomorphic
  map to $\C \to X \qu K$.
\item {\rm (Affine vortex)} If $\lim_{\nu \to \infty} \eps_\nu \in
  (0,\infty)$, then $\phi_\nu^* (A_\nu,u_\nu)$ converges to a vortex
  on the affine line $\bA$, with respect to the Euclidean area form
  $\Vol_\bA = \frac{i}{2} \d z \wedge \d \ovl{z}$.
 \item {\rm (Sphere Bubble in $X$)} If $\lim_{\nu \to \infty} \eps_\nu = 0
   $, then $\phi_\nu^* (A_\nu,u_\nu)$ converges to a vortex with zero
   vortex parameter, that is, a holomorphic map $\C \to X$.
\end{enumerate} 
\end{proposition} 

The bubble trees that occur are described further in the following
section. 

\subsection{Affine symplectic vortices}

In this section we further study vortices on the complex affine line
$\bA$, which arose in the Gaio-Salamon study \cite{ga:gw} of the
large area limit.  Ziltener \cite{zilt:phd} studied the bubbles that
arise in more detail.

\begin{definition} {\rm (Affine vortices)}
Suppose that $X$ is a compact Hamiltonian $K$-manifold equipped with
an invariant almost complex structure $J$.  An {\em $n$-marked affine
  symplectic vortex} with target $X$ is a datum $(A,u,\ul{z})$, where
$A \in \Omega^1(\bA,\g)$ is a connection on the trivial bundle $P :=
\bA \times K \to \bA$, $u: \bA \to X$ is a $J_A$-holomorphic map,
$\ul{z} = (z_0,\ldots, z_n) \in \bA^n$ is an $n$-tuple of distinct
points, and
$$ {\rm (Affine \ Vortex \ Equation)} \quad F_A + u^* \Phi \, \Vol_\bA = 0. $$
An {\em isomorphism} of scaled vortices $(A_0,u_0)$ to $(A_1,u_1)$ is
translation $\phi$ of $\bA$ and a gauge transformation $k: \bA \to K$
satisfying $k^* \phi^* (A_1,u_1) = (A_0,u_0)$.
\end{definition} \vskip .1in 

Let $M_{n,1}^K(\bA,X)$ denote the moduli space of isomorphism classes
of $n$-marked affine vortices.  It has a natural compactification by
isomorphism classes of {\em nodal affine vortices} described as
follows.  Let $C$ be an $(n+1)$-marked connected genus zero nodal
curve with irreducible components $C_1,\ldots, C_k$.  For each
irreducible component $C_i$ not containing $z_0$ there is a unique
node $\hat{w}_i$ which disconnects $C_i$ from the marking $z_0$; we
denote by $C^\circ_i = C_i - \{ \hat{w}_i \}$ the complement.  In case
$C_i$ contains $z_0$ we denote $C^\circ_i = C_i - \{ z_0 \}$.  Each
affine curve $C^\circ_i$ is isomorphic to $\C$, uniquely up to
translation and dilation.  Thus $C^\circ_i$ admits a unique
equivalence class of K\"ahler forms, equal to $\omega_{\bA} = \d z
\wedge \d \ovl{z} (i/2) $ up to scalar multiplication.

\begin{definition}  {\rm (Nodal affine vortices)} 
Suppose that $X$ is a compact Hamiltonian $K$-manifold equipped with
an invariant almost complex structure $J$.  An {\em $n$-marked nodal
  affine symplectic vortex} with target $X$ is a datum
$(C,P,A,u,\omega,\ul{z})$ consisting of a connected $(n+1)$-marked nodal
curve $C$ together with a principal $K$-bundle $P \to C$, a (possibly
infinite or zero) two form $\omega: C \to \P(\Lambda^2 T^\dual_\R C
\oplus \R)$, a connection $A_i$ on each $P |C_i$, a section $u: C \to
P(X)$, and markings $\ul{z}= (z_1,\ldots, z_n)$, such that each
irreducible component $C_i$ of $C$ is one of the following three types:
\begin{enumerate}
\item (Zero Scaling) Components with zero two-form, equipped with a
  trivial bundle $P | C_i$ and a pseudoholomorphic map $u_i: C_i \to
  X$.
\item (Finite, non-zero scaling) Components with non-zero, finite area
  form $\omega | C_i \in \Omega^2(C^\circ_i)$ equal to a non-zero
  multiple of $\omega_{\bA}$ on $C^\circ_i \cong \bA$ and an affine
  vortex $(A_i,u_i)$ on $C^\circ_i$.
\item (Infinite scaling) Components $C_i$ with infinite two-form
  $\omega | C_i$, equipped with holomorphic sections $u| C_i : C_i \to
  P(X)$ mapping to the zero level set $P(\Phinv(0))$, and so defining
  a holomorphic map $\ovl{u}_i: C_i \to X \qu K.  $
\end{enumerate} 
This datum should satisfy the following conditions
\begin{enumerate} 
\item (Monotonicity) For every non-self-crossing path from a marking
  $z_i, i > 0$ to the marking $z_0$, the path crosses exactly one
  irreducible component with finite, non-zero area form, and all
  irreducible components before resp. after that irreducible component
  have zero resp. infinite area form.
\item (Continuity) If $C_i$ meets $C_j$ at a node represented by a
  pair $(w_{ij}, w_{ji}) \in C_i \times C_j$ then $u_i(w_{ij}) =
  u_j(w_{ji})$.
\item (Stability) If $C_i$ is an irreducible component on which the two-form is
  zero or infinity resp. finite and non-zero and $u_i$ is constant
  resp. $u_i$ is covariant constant and $A_i$ is flat then $C_i$
  contains at least two resp. three special points.
\end{enumerate} 
\end{definition} \vskip .1in 

\begin{remark} An irreducible component $C_i$ is a {\em ghost component} if it
satisfies one of the hypotheses requiring at least three special
points or non-degenerate scalings; that is, $ \ovl{u}_i$ is constant or
$A | C_i$ is flat and $u |C_i$ is covariant constant.  The stability
condition can then be reformulated as the condition that any ghost
component has at least three special points (nodes or markings) or two
special points and a non-zero, finite area form.  Either of these
conditions is equivalent to the absence of non-trivial infinitesimal
automorphisms: infinitesimal automorphisms arising from gauge
transformations are impossible because of the local freeness
assumption for the action on the zero level set of the moment map. 
\end{remark} \vskip .1in 

There is a natural notion of {\em convergence} of affine nodal
vortices which generalizes convergence with fixed area form in
Definition \ref{vortexconverge}.

\begin{definition} {\rm (Convergence of nodal affine vortices)} 
A sequence of isomorphism classes of nodal marked affine symplectic
vortices $[(C_\nu,P_\nu,A_\nu,u_\nu,\omega_\nu,\ul{z}_\nu)]$ {\em
  converges} to a nodal marked affine vortex
$[(C,P,A,u,\omega,\ul{z})]$ iff there exists for each irreducible
component $C_j$ of $C$, a sequence $U_{j,\nu} \subset C_j$ of
increasing open neighborhoods and for each $\nu$, a holomorphic
embedding $\phi_{j,\nu}: U_{j,\nu} \to C_\nu$, an isomorphism
$\psi_{j,\nu}: \phi_{j,\nu}^* P_\nu \to P$ such that if
$\psi_{j,\nu}(X): \phi_{j,\nu}^* P_\nu(X) \to P(X)$ denotes the
associated maps of fiber bundles then 
\begin{enumerate} 
\item {\rm (Open neighborhoods cover)} the union of
  $\phi_{j,\nu}(U_{j,\nu})$ is $C_\nu$;
\item {\rm (Marked curves converge)} If $z_i \in C_j$ then the limit
  of $\phi_{j,\nu}^{-1}(z_{\nu,i})$ is defined and equal to $z_i$.
  Furthermore, if $i \neq j$ then the image of $\phi_{j,\nu}^{-1}
  \circ \phi_{i,\nu}$ converges to the node of $C_i$ connecting to
  $C_j$;
\item {\rm (Connections converge)} $\psi_{j,\nu}^* A_\nu$ converges to
  $A | C_j$ uniformly in all derivatives in any compact subset of 
  any $U_{j,\nu'}$;
\item {\rm (Sections converge)} $\psi_{j,\nu}(X)^* u_\nu$ converges to
  $u | C_j$ uniformly in all derivatives in any compact subset of any 
  $U_{j,\nu}$;
\item {\rm (Scalings converge)} $\phi_{j,\nu}^{*} \omega$ converges to
  $\omega | C_j$ in $C^0$ on any compact subset of any $U_{j,\nu'}$;
\item {\rm (Energies converge)} $ \lim_{\nu \to \infty}
  E(C_\nu,P_\nu,A_\nu,u_\nu,\omega_\nu,\ul{z}_\nu)) =
  E(C,P,A,u,\omega,\ul{z}) .$
\end{enumerate} 
\end{definition} \vskip .1in  

\begin{proposition} {\rm (Sequential compactness for affine vortices)}
\label{affinecompact} Suppose that $X$ is a Hamiltonian $K$-manifold with a proper
moment map convex at infinity and an invariant almost complex
structure $J$.  Any sequence of isomorphism of classes of stable
marked affine vortices with bounded energy has a convergent
subsequence.
\end{proposition} 

The proof of the theorem above combines results of Ziltener
\cite{zilt:phd}, \cite{zilt:qk} who discusses bubbling of affine
vortices and bubbles in $X \qu K$, and Ott \cite{ott:remov}, who
treats bubbling in $X$.   Namely:
\begin{theorem} {\rm (Bubbling analysis for affine vortices)}  
\label{affinevortices} \label{annlem2}
Suppose that $X$ is a Hamiltonian $K$-manifold with proper moment map
convex at infinity equipped with an invariant almost complex structure
$J$.
\begin{enumerate} 
\item {\rm (Energy quantization)} \cite[Lemma D.1]{zilt:phd} There
  exists a constant $E_0 > 0$ such that any non-trivial symplectic
  vortex $(A,u) $ on $\C$ satisfies $E(A,u) > E_0$.
\item \label{al} {\rm (Annulus Lemma)} \cite[Lemma 4.11]{zilt:phd} For every
  compact subset $X_0 \subset X$ and every number $r_0 > 0$ there are
  constants $E_1 > 0 , a >0 $ and $c_1 > 0 $ such that the following
  holds.  Assume that $r_0 \leq r < R \leq \infty$ and $(A,u)$ is a
  vortex on the annulus $A(r,R) = B_R(0)- B_r(0)$ such that $u(z) \in
  X_0$ for every $z \in A(r,R)$, and suppose that $E( (A,u) | A(r,R))
  \leq E_1$.  Then for every $\rho \ge 2$ we have
$$ E( (A,u) |_{A( \rho r, \rho^{-1}R)}) \leq c_1 E(A,u) \rho^{-a} .$$
\item \label{mvt} {\rm (Mean value inequality)} Let $X_0 \subset X$ be
  a compact subset.  Then there exists a constant $E_0 > 0$ such that
  for every $z_0 \in \C$, $r> 0$ and every symplectic vortex $(A,u)$
  satisfying $u(B_r(z_0)) \subset X_0$ and $E( (A,u) | B_r(z_0)) \leq
  E_0$, the energy density $e_{A,u}$ given by the integrand in
  \eqref{energy} satisfies the estimate
$$ e_{A,u}(z_0) := \hh \Vert\d_A u(z_0)\Vert^2 + \Vert \Phi(u(z_0)) \Vert^2
\leq (8/\pi r^2) E((A,u) | B_r(z_0)) .$$
\item {\rm (Removal of singularities)} \cite[Proposition
  D.6]{zilt:phd} \label{k0} Let $(A,u)$ be a finite energy vortex on $\C$.  The
  map $Ku: \C \to X/K$ extends continuously to a map $\P \to X/K$,
  such that $Ku(\infty) \in X \qu K$.  Furthermore,
\begin{enumerate} 
\item there are constants $E > 0, C> 0$ and
  $\delta > 0$ such that the following holds.  For every vortex
  $(A,u)$ on $\C$ and every $R \ge 1$ such that $E(w, \C \bs B_R) < E$
  and every $z \in \C \bs B_{2R}$ we have $ e_{A,u}(z) \leq CR^\delta
  |z|^{-2 - \delta}$;
\item there exist a number $\delta > 0$ such that for $2 \leq p < 4/(2
  - \delta)$, then
$$ x_0 := \lim_{r \to \infty} u(r,0) $$
exists, and there exists a map $k_0 \in W^{1,p}([0,2 \pi ],G)$ such
that if $A_\theta(r)$ denotes the restriction of the connection $A$ in
radial gauge to the circle $\{ |z| = e^r \} \cong S^1$, then
$$ \lim_{r \to \infty} \max_{\theta \in S^1} \d( u(re^{i \theta}),
k_0(\theta) x_0) = 0 .$$
$$ \sup_{ r \ge 0} \Vert \partial_\theta k_0 k_0^{-1} + A_\theta(r)
  \Vert_{L^p(S^1)} e^{(-1 + 2/p + \delta/2) r} < \infty .$$
Necessarily $x_0$ is fixed by $k_{2\pi}$, which since $K$ is compact
and acts locally freely on $\Phinv(0)$, is finite order.
\end{enumerate}
\end{enumerate}  
\end{theorem} 

\begin{remark}  \label{extend}
The compactness argument for some of the moduli spaces below uses a
slightly more general argument than that stated in Ziltener
\cite{zilt:qk} and Ott \cite{ott:remov}.  Namely, multiples of the
standard area form on the projective line gives rise to area forms in
local coordinates
\begin{equation} \label{resc} 
\lambda(s,t) \ \d s \wedge \d t = (1 + \eps(s^2 + t^2))^{-2} \ \d s
\wedge \d t \end{equation}
for $\eps > 0$.  We will need the (Annulus Lemma) \eqref{al} above and
the associated exponential decay of the derivative of a vortex for
this area form.  To see that the result still holds, we remark that
Ziltener's results on the invariant symplectic action in
\cite{zil:decay} hold for any area form $ \lambda \d s \wedge \d t $
satisfying the following property: Let
$$ m = \inf \{ \Vert \xi_X(x) \Vert \ | \ x \in X, \ \Vert \xi \Vert = 1 \} $$
denote the minimal length of generating vector fields for Lie algebra
vectors of length one.  We say that $\lambda(s,t) \d s \d t$ is {\em
  admissible} for some constant $a > 0 $ iff
\begin{equation} \label{admissible}  \lambda \ge \frac{2\pi}{am}, \quad \sup | d( \lambda^{-1}) |^2 + | \Delta
(\lambda^{-2})| < 2m^2 \end{equation}
see Ziltener \cite[Equation 1.7]{zil:decay}, in particular, for the
area form \eqref{resc}.  The mean value inequality in Ott
\cite[Section 2]{ott:remov} holds uniformly in $\eps$ for the family
of area forms \eqref{resc}.  Thus the inequality in the (Annulus
Lemma) implies an inequality for the covariant derivative of the map
away from the ends.
\end{remark} \vskip .1in 

In the case that $K$ acts freely on $\Phinv(0)$, any finite energy
vortex $(A,u)$ on $\C$ has a well-defined {\em evaluation at infinity}
$\ev_\infty(A,u) \in X \qu K$, given by the limit of $u(s + i t)$
along any ray $ s + i t = re^{i \theta}, r \to \infty$.  More
generally in the locally free case, $\ev_\infty(A,u) = [k_0(2
  \pi),x_0]$ lies in the {\em inertia orbifold}
$$ I_{X \qu K} \in \{ (k,x) \in K \times \Phinv(0) \ | \ kx = x \} /K $$
which, for example, appeared in Kawasaki \cite{ka:ri}.  The orbifold
case was not treated in Ziltener \cite{zilt:phd}; however, the proof
is almost the same as the manifold case.

\begin{remark}   \label{extends}
{\rm (Failure of removal of singularities)} The area form
$\omega_{\bA}$ has a pole 
at infinity and so one cannot expect an extension of $(A,u)$ to the
projective line to satisfy the vortex equations. Instead, taking $P$
to be a orbi-bundle on $\P^1$ given by gluing in the trivial bundle
using transition map $z^\lambda$, the pair $(A,u)$ extends to a gauged
holomorphic map on $\P$ mapping $\infty$ to $\Phinv(0)$
\cite{venuwood:class}.  However, the extension will not satisfy any
vortex equation on the entire projective line $\P$.
\end{remark}  \vskip .1in


Let $\ovl{M}_{n,1}^K(\bA,X)$ resp. $\ovl{M}_{n,1}^{K,\fr}(\bA,X)$ denote the
moduli space of isomorphism classes of nodal scaled affine vortices to
$X$, resp. the moduli space of isomorphism classes of framed nodal
scaled affine vortices to $X$.  $\ovl{M}_{n,1}^\fr(\bA,X)$ admits an
evaluation maps at the markings, and, as explained in \cite{zilt:phd}
an additional evaluation map at infinity to $I_{X \qu K}$:
$$ \ev^\fr \times \ev_\infty : \ovl{M}^{K,\fr}_{n,1}(\bA,X) \to X^{n}
\times I_{X \qu K} .$$
If $K^{n}$ acts freely, combining this map with a classifying map
gives a map
$$ \ev \times \ev_\infty: \ovl{M}_{n,1}^K(\bA,X) \to X_K^{n} \times I_{X
  \qu K} .$$
If $K^n$ only acts locally freely, then the map above exists as a
morphism of stacks, or after passing to a classifying space for the
groupoid $\ovl{M}_{n,1}^K(\bA,X)$, which has the same rational
cohomology as $\ovl{M}_{n,1}^K(\bA,X)$.

\begin{remark} The case that $X$ and $K$ are trivial gives the complexified
multiplihedron $\ovl{M}_{n,1}$ constructed in Section \ref{ziltener}.
Indeed, each irreducible component $C_i$ with finite or non-zero scaling is
equipped with a isomorphism with the affine line, unique up to
translation, and therefore a scaling $\lambda_i$.  Thus the underlying
curve of any affine vortex is automatically a (possibly unstable)
scaled affine curve in the sense of Section \ref{ziltener}.
\end{remark}  \vskip .1in



Let $M_{n,1}^K(\bA,X)$ denote the moduli space of isomorphism classes
of finite energy $n$-marked vortices on $\bA$ (the additional marking
at infinity) with values in $X$.  By combining the sequential
compactness theorem \ref{affinecompact} with local distance functions
as in \eqref{distfns}, one has:

\begin{theorem} {\rm (Properness of the moduli space of affine vortices)}
Suppose that $X$ is a Hamiltonian $K$-manifold with proper moment map
convex at infinity, equipped with an invariant compatible almost
complex structure.  $\ovl{M}_{n,1}^K(\bA,X)$ is Hausdorff and the
energy map $E: \ovl{M}_{n,1}^K(\bA,X) \to [0,\infty)$ is proper.
\end{theorem}  

\subsection{Stable maps to orbifolds} 

In order to understand bubbling in the infinite area limit we briefly
review the definition of a stable maps for orbifold targets, studied
in Chen-Ruan \cite{cr:orb}, which appear as bubbles in the definition
of certain gauged maps.  The corresponding algebraic theory by
Abramovich-Graber-Vistoli \cite{agv:gw} will be reviewed in
\cite[Section 5]{qk2}.  Recall that an {\em orbifold structure} on a
topological space $C$ is a proper \'etale groupoid together with a
homeomorphism of the space of objects to $C$. An {\em nodal orbifold
  structure} on a nodal complex curve $C$ is an orbifold structure on
its normalization, such that the automorphism group of any nodal point
is independent of the choice of irreducible component containing it.
A nodal orbifold structure is a {\em twisting} if each point with
non-trivial automorphism is either a node or marking, and for each
node, the orbifold charts satisfy the following {\em balanced
  condition}: the chart on one side of the node is of the form
$U/\mu_r$ for some $r$ where $\mu_r$ acts on $U \subset C$ a
neighborhood of $0$, while the orbifold structure on the other side is
$U/\mu_r$ with the conjugate action by $\exp( - 2 \pi i /r)$.

Let $Y$ be a compact symplectic orbifold equipped with a compatible
almost complex structure $J$.  A {\em stable map} to $Y$ is a complex
nodal curve $C$ equipped with a twisting and a representable
pseudoholomorphic orbifold map $u: C \to Y$. Representability means
that the map $u$ is smooth after passing to \'etale cover in $Y$: If
$z \in C$ and $u(z) \in Y$ has automorphism group $\Aut(u(z))$, then
there exists an orbifold chart $U/\Aut(u(z))$ for $Y$ near $u(z)$ so
that the groupoid fiber product $ u^{-1}(U/\Aut(z))
\times_{U/\Aut(u(z))} U$ is an orbifold chart for $z$ and $u$ has a
smooth local lift $\ti{u}$ to $U$ given by projection on the second
factor.  In particular, $\Aut(z)$ injects into $\Aut(u(z))$.  The
notion of Gromov convergence of twisted pseudoholomorphic maps which
generalizes Gromov convergence of stable pseudoholomorphic maps.  Let
$\ovl{M}_{g,n}(Y)$ denote the space of isomorphism classes of
connected, genus $g$, $n$-marked stable maps to $Y$.

\begin{theorem} \label{orbcompact}  {\rm (Properness of moduli spaces of stable 
maps to orbifolds)} Let $Y$ be a compact symplectic orbifold equipped
  with a compatible almost complex structure.  $\ovl{M}_{g,n}(Y)$ is
  Hausdorff and the energy map $E: \ovl{M}_{g,n}(Y) \to [0,\infty)$ is
    proper.
\end{theorem} 

A proof of this theorem is sketched by Chen-Ruan \cite{cr:orb}, who
list the necessary changes in the proof of the compactness of the
moduli space of stable maps for manifold targets.

\subsection{Vortices with varying scaling} 

The adiabatic limit Theorem \ref{largearea} will be proved by studying
a moduli space of curves with {\em varying scaling} which interpolates
between the moduli space of vortices for a fixed area form $\Vol_C$
and the infinite area limit.  More precisely, Theorem \ref{largearea}
will follow from a divisor class relation in the source moduli space
constructed in Section \ref{compose}, just as associativity of the
quantum product follows from a divisor class relation in the moduli
space of stable curves.

First we construct the objects that appear in the infinite area
limit. Let $X$ be a Hamiltonian $K$-manifold with proper moment map,
so that $X \qu K$ is a locally free quotient and so an orbifold.

\begin{definition} {\rm (Nodal infinite area vortices)}   An $n$-marked {\em nodal infinite area vortex} consists of the following
datum:
\begin{enumerate} 
\item 
(Stable map to the quotient) a stable $r$-marked pseudoholomorphic map
  $u: C_0 \to C \times X \qu K$ of class $([C],d_0)$ for some $d_0 \in
  H_2(X \qu K)$;
\item (Affine vortex bubbles) for each marking $z_j \in C_0$ a stable
  $i_j$-marked affine vortex $(C_j,A_j,u_j,\ul{z}_j)$ with markings
  $\ul{z}_j = (z_{j,1},\ldots, z_{j,i_j})$ and orbifold structure of
  $C_j$ at the point at infinity $z_{j,0}$ matching the orbifold
  structure of $C_0$ at $z_{0,j} \in C_0$, so that the union $C_0 \cup
  C_1 \cup \ldots \cup C_r$ is a balanced orbifold curve;
\item (Matching condition) the value $u_i(z_{i,0}) =
  p_2(u_i(z_{i,0}))$ in $X \qu K$, where $p_2: C \times X \qu K \to X
  \qu K$ is the projection on the second factor.
\end{enumerate}
 An {\em isomorphism} of nodal infinite area vortices is a combination
 of an automorphism of the underlying curves intertwining the scalings
 and a bundle isomorphism on the affine vortex bubbles, intertwining
 the markings, scalings, connections and maps.
\end{definition} \vskip .1in 

\begin{remark}  We denote by $C$ the nodal curve obtained by gluing together the
curves $C_0$ and $C_i$ at $(z_i,z_{0,i})$; the matching condition
means that the orbifold structures glue together to a twisting of $C$.
Removal of singularities for vortices in Remark \ref{extends} implies
that orbi-bundles $P_j$ defined by the limit of the connections at
infinity glue together with the bundle $P_0 \to C_0$ given by
pull-back of $\Phinv(0) \to X \qu K$, to give a bundle $P \to C$.  The
matching condition then implies that the section $u_0$ over $C_0$ and
$u_j$ over $C_j$ glue together to a section of $P(X)$ over $C$ with
fairly weak regularity properties at the nodes with infinite scaling,
so that $u$ takes values in the zero level set $\Phinv(0)$ on the
subset of $C$ with infinite scaling.  An isomorphism of nodal
symplectic vortices is then an automorphism of the domain intertwining
the connections (singular at the nodes with infinite scaling), maps,
scalings, and markings.
\end{remark} \vskip .1in 

Let $\ovl{M}_n^K(C,X)_\infty$ denote the moduli space of isomorphism
classes of nodal infinite-area vortices.  From the description above,
$\ovl{M}_n^K(C,X)_\infty$ is the union of fiber products
$$ \ovl{M}_r^K(C,X)_\infty \times_{(I_{X \qu K})^r} \prod_{j=1}^r
\ovl{M}_{|I_j|,1}^K(\bA,X) $$
over unordered partitions $[I_1,\ldots,I_r]$ of $\{1,\ldots, n\}$.

We combine the moduli spaces above into a moduli space of vortices
with varying area form.  Let $\Vol_C \in \Omega^2(C,\R)$ be an area
form.  

\begin{definition} 
\begin{enumerate} 
\item {\rm (Two-form corresponding to a scaling)} Any scaled curve
  $(v: \hat{C} \to C, \lambda: \hat{C} \to \P( T^\dual_v \oplus \C))$
  in the sense of Definition \ref{scaledcurves} defines a volume form
  $\Vol_C(v,\lambda) \in \Omega^2(\hat{C},\R \cup \{ \infty \})$ as
  follows.
\begin{enumerate} 
\item If $\lambda$ is finite on the principal component, then
  $\Vol_C(v,\lambda)$ is $|\lambda|^2$ times the form $\Vol_C$ on the
  principal component, and on the bubble component $\Vol_C(v,\lambda)$
  vanishes.
\item If $\lambda$ is infinite on the principal component, then
  $\Vol_C(v,\lambda)$ is equal to $\lambda \wedge \ovl{\lambda}$ on
  each bubble tree.
\end{enumerate} 
\item  
 {\rm (Scaled vortices)} A {\em marked scaled vortex} with domain $C$
 and target $X$ consists of a marked scaled curve $(v: \hat{C} \to C,
 \lambda: \hat{C} \to \P( T^\dual_v \oplus \C))$ together with a
 vortex on $\hat{C}$ corresponding to the area form
 $\Vol_C(v,\lambda)$.  That is,
\begin{enumerate} 
\item If $\lambda$ is finite on the principal component, then a vortex
  $v_0$ on the principal component corresponding to the area form
  $\Vol_C(v,\lambda) | C_0$ and a collection of sphere bubbles
  $v_i:C_i \to X$ for each component $C_i \subset \hat{C}, i \neq 0$;
\item If $\lambda$ is infinite on the principal component, then a
  holomorphic map $u_0: C_0 \to X \qu K$ on the principal component
  $C_0$ and a collection of nodal affine vortices on the bubble trees
  attached to $C_0$;
\end{enumerate} 
Both should satisfy natural matching conditions at the nodes.  A
marked scaled vortex is {\em polystable} if each non-principal
component with finite, non-zero scaling (resp. zero or infinite
scaling) has at least $2$ (resp. $3$) special points.  The notion of
{\em isomorphism} of marked scaled vortices, using gauge
transformations and automorphisms of the domain, is left to the
reader.  A marked scaled vortex is {\em stable} if it is polystable
and has no infinitesimal automorphisms.
\end{enumerate} 
\end{definition} \vskip .1in 

We denote by $\ovl{M}_{n,1}^K(C,X)$ the moduli space of isomorphism
classes of marked scaled vortices, and by $\ovl{M}_{n,1}^K(C,X,d)$ the
component with homology class $d \in H_2^K(X,\Z)$.  The notion of
convergence of vortices extends naturally to the case of varying
scaling, and defines a topology on $\ovl{M}_{n,1}^K(C,X)$.

\begin{theorem}   \label{scaledproper} 
 {\rm (Properness of the moduli space of scaled symplectic vortices)}
 Let $X$ be a Hamiltonian $K$-manifold with proper moment map convex
 at infinity equipped with a compatible almost complex structure, such
 that $X \qu K$ is a locally free quotient.  Then
 $\ovl{M}_{n,1}^K(C,X)$ is Hausdorff and the energy map $E:
 \ovl{M}_{n,1}^K(C,X) \to [0,\infty)$ is proper.
\end{theorem} 

The proof depends on an extension of the results of Ziltener
\cite{zilt:qk}, \cite{zil:decay} and Ott \cite{ott:remov} to the case
of vortices with varying scaling.  We need the following extension of 
\cite[Theorem 1.3]{zil:decay}:

\begin{lemma} \label{bconnect} 
For every constant $\eps > 0 $, there exist constants $C,\delta > 0$
such that if $(A,u)$ is a vortex on the cylinder $\Sigma = [-S,S]
\times S^1$ with respect to an admissible area form as in
\eqref{admissible} with respect to some constant $a$ with $E(A,u) <
\delta$, then the energy density $e_{A,u}$ satisfying for $(s,t) \in
      [-S + 1, S - 1] \times S^1$
$$e_{A, u}(s,t) < C E(A,u) \lambda^{-2} e^{(-2 \pi/a + \eps)s} .$$
\end{lemma} 

\begin{proof}  We assume that $(A,u),\eps,a$ are as in the statement 
of the Lemma. As in Ziltener \cite[p.17]{zil:decay}, the assumptions
imply that for $\delta$ sufficiently small, the invariant symplectic
action $\cA$ in \cite{zil:decay} is well-defined and satisfies the
energy-action inequality
$$ E((A,u)|_{[s,s'] \times S^1}) = - \cA((A,u)(s', \cdot)) +
\cA((A,u)(s, \cdot)) .$$
Then as in \cite[p. 46]{zilt:qk}, 
$$ \dds E((A,u)|_{[S + s, S - s] \times S^1})
 \leq -(2\pi/a - \eps)
E((A,u)|_{[S + s, S - s] \times S^1})
 $$
which implies the necessary exponential decay for the energy.  The
exponential decay for the energy density follows from the point-wise
estimate in \cite[Lemma 3.3]{zil:decay}: If $(A,u)$ has sufficiently
small energy on a ball of radius $1$ around $(s,t)$ contained in
$[-S,S] \times S^1$ then the energy density satisfies
$$e_{A,u}(s,t) \leq \frac{32}{\pi} E((A,u)_{B_1(s,t)}) .$$
Note that the statement of Lemma \cite[Lemma 3.3]{zil:decay} is only
for a fixed vortex on a half-cylinder, on a sufficiently small
neighborhood of infinity, but an examination of the short proof shows
that it suffices to have sufficiently small energy on the ball.
\end{proof}

\begin{proof}[Proof of Theorem \ref{scaledproper}]  We claim sequential compactness,      
that is, that any sequence of vortices $(\hat{C}_\nu,A_\nu,u_\nu)$
with varying scaling $\lambda_\nu$ with bounded energy has a
convergent subsequence.  For sequences with scaling bounded away from
infinity, this follows from Ott \cite{ott:remov}.  For scalings
$\lambda_\nu$ approaching infinity, one obtains convergence in the
absence of bubbling by the arguments of Gaio-Salamon \cite{ga:gw}, see
Theorem \ref{sc}.  That reference also shows that bubbles in $X \qu K$
or $X$ or affine vortices satisfy energy quantization. Hence bubbling
(blow-up of the first derivative of the section) happens only at
finitely many points.  It follows that after passing to a subsequence,
$(\hat{C}_\nu,A_\nu,u_\nu)$ converges on the complement of a finite
subset of the principal component, an infinite area vortex.  By
Proposition \ref{bij} such an infinite area vortex corresponds to a
pseudoholomorphic map to $X \qu K$.  By removal of singularities, one
obtains a stable map to $X \qu K$.  The annulus lemma in parts (c,d)
of Theorem \ref{annlem2}, and the extension described in Lemma
\ref{bconnect}, implies that bubbles in $X$ connect with the affine
vortices.  The standard soft rescaling argument (see for example
\cite[Section 2.6]{zilt:qk}) shows the sequential compactness
statement.  Let the closed sets be those for which any convergent
sequence has a limit.  That this defines a topology, and that this
topology is Hausdorff, uses the construction of local distance
functions as in \eqref{distfns} which will be left to the reader.
\end{proof} 


\def\cprime{$'$} \def\cprime{$'$} \def\cprime{$'$} \def\cprime{$'$}
  \def\cprime{$'$} \def\cprime{$'$}
  \def\polhk#1{\setbox0=\hbox{#1}{\ooalign{\hidewidth
  \lower1.5ex\hbox{`}\hidewidth\crcr\unhbox0}}} \def\cprime{$'$}
  \def\cprime{$'$} \def\cprime{$'$} \def\cprime{$'$}

\end{document}